\documentclass[11pt,leqno]{article}
\usepackage{amsmath,amsthm,amssymb,amscd,url}
\usepackage[all]{xy}
\usepackage{enumerate}
\usepackage{comment}
\usepackage{tikz}
\usepackage{tikz-qtree}
\usepackage{verbatim}
\usepackage{lscape}
\usepackage{color}
\usepackage{marginnote}
\usepackage{fullpage}

\usepackage{hyperref}
\hypersetup{%
 colorlinks=true,
 linkcolor=blue,
 filecolor=magenta, 
 urlcolor=blue,
}
\author{}
\title{}
\date{}

\newtheorem{lemma}{Lemma}[section]
\newtheorem{conjecture}[lemma]{Conjecture}

\newtheorem{proposition}[lemma]{Proposition}
\newtheorem{theorem}[lemma]{Theorem}

\newtheorem{appthm}{Theorem}

\newtheorem*{corollary-nonum}{Corollary}
\newtheorem*{lemma-nonum}{Lemma}

\theoremstyle{definition}
\newtheorem{remark}{Remark}[section]
\newtheorem*{remark-nonum}{Remark}

\newcommand{\bbQ}{{\mathbb Q}}
\newcommand{\bbR}{{\mathbb R}}
\newcommand{\bbZ}{{\mathbb Z}}

\newcommand{\rS}{\mathrm{S}}

\newcommand{\labthm}[2]{{\footnotesize \fbox{\makebox[0.25\width][c]{
                               \begin{tabular}{c}
                               #1 \\ \hspace{3pt}#2
                               \end{tabular}     }
                               }
                         }      
                       }

\numberwithin{equation}{section}
\numberwithin{table}{section}

\begin{document}

\title{Near-squares in binary recurrence sequences}
\author{Nikos Tzanakis \thanks{e-mail: {\tt tzanakis@uoc.gr}}
        \\
        Department of Mathematics \& Applied Mathematics \\
        University of Crete, Greece  
\and 
Paul Voutier \thanks{e-mail: {\tt paul.voutier@gmail.com}}
\\
London, United Kingdom
}
\maketitle

%%%%%%%%%%%%%%%%%%% Begin: Abstract for arXiv %%%%%%%%%%%%%%%%%%%%
\begin{abstract}
We call an integer a \emph{near-square} if its absolute value is a square
or a prime times a square. We investigate such near-squares in the binary
recurrence sequences defined for integers $a \geq 3$ by $u_{0}(a)=0$,
$u_{1}(a)=1$ and $u_{n+2}(a)=au_{n+1}(a)-u_{n}(a)$ for $n \geq 0$.
We show that for a given $a \geq 3$, there is at most one $n \geq 5$ such that
$u_{n}(a)$ is a near-square. With the exceptions of $u_{6}(3)=12^{2}$ and
$u_{7}(6)=239 \cdot 13^{2}$, any such $u_{n}(a)$ can be a near-square only if
$a \equiv 2 \pmod{4}$, $n \equiv 3 \pmod{4}$ is prime and $n \geq 19$.

This is a part of a more general phenomenon regarding near-squares in non-degenerate
recurrence sequences defined for integers $a$ and $b=-b_{1}^{2}$ by $u_{0}(a,b)=0$,
$u_{1}(a,b)=1$ and $u_{n+2}(a,b)=au_{n+1}(a,b)+bu_{n}(a,b)$ for $n \geq 0$. This arises
from a novel Aurifeuillean type factorisation of $u_{2n+1}(a,b)$ we have found.
%%%%%%%%%%%%%%%%%%% End: Abstract for arXiv %%%%%%%%%%%%%%%%%%
\end{abstract}

\tableofcontents

\section{Introduction}
\label{sec:intro}

The study of recurrence sequences and their arithmetic properties has a long history,
going back to at least Fibonacci. Well-known questions and conjectures include 
whether there are infinitely many primes in such sequences, their divisibility properties
and also the existence of perfect powers in such sequences.

In the early 1960's, Cohn and Wyler \cite{Cohn1,Cohn2,Wy} independently proved
that $u_{0}=0$, $u_{1}=u_{2}=1$ and $u_{12}=12^{2}$ are the only squares in the
Fibonacci sequence. 
Just over 40 years later, Bugeaud, Mignotte and Siksek \cite{BMS} proved that
$u_{0}=0$, $u_{1}=u_{2}=1$, $u_{6}=2^{3}$ and $u_{12}=12^{2}$ are the only perfect
powers in the Fibonacci sequence, as well as proving that $u_{1}=1$ and $u_{3}=2^{2}$
are the only perfect powers in the associated Lucas sequence (defined by 
$u_{0}=2$, $u_{1}=1$, $u_{n+2}=u_{n+1}+u_{n}$).
Many other results on powers and ``near'' powers for binary recurrence sequences
have been obtained too.
In particular, let us mention a result of Mignotte and Peth\H{o} \cite{MP} that
is relevant to our work here. To state
their result, let us first establish some notation.

Let $a>0$ and $b \neq 0$ be two relatively prime integers with $\Delta=a^{2}+4b \neq 0$,
and put
\begin{align}
\label{eq:seq-defn}
\alpha &=\frac{a+\sqrt{\Delta}}{2},\quad \beta=\frac{a-\sqrt{\Delta}}{2} 
\quad
\text{(roots of $X^{2}-aX-b$),}\\
u_{n}(a,b) &= u_{n}:=\frac{\alpha^{n}-\beta^{n}}{\alpha-\beta},\quad 
u_{0}=0, \,\, u_{1}=1, \,\, u_{n+2}=au_{n+1}+bu_{n}. \nonumber
\end{align}

A binary recurrence sequence is called \emph{non-degenerate} if $\alpha/\beta$
is not a root of unity, and $a$ and $b$ are non-zero relatively prime
integers.

We put $\rS= \left\{ x^{2}: x \in \bbZ \right\}$ and, for any non-zero integer, $c$,
let $c\rS=\{ cx: x \in \rS \}$. We say that an integer is a \emph{near-square}
if it is an element of $c\rS$ with $|c|=1$ or $|c|$ a prime.

Mignotte and Peth\H{o} \cite{MP} proved that, for $a>3$ and 
$n>3$, $u_{n}(a,-1)$ is never a square, or $2$, $3$ or $6$ times a square, except when 
$a=338$ and $n=4$.

Our results here extend those of Mignotte and Peth\H{o}. 
We investigate the problem of when elements $u_{n}(a,-1)$ are near-squares.
What is significant about such a result is that it is uniform in $p$. This is
not what one expects in general. For example, with the Fibonacci sequence
and the Mersenne numbers, one expects there are infinitely many primes and hence
infinitely elements of
the form $px^{2}$.
In fact, what we prove here is a subset of what we believe to be true.

\begin{conjecture}
\label{conj:new1}
Let $a$ and $b_{1}$ be relatively prime positive integers and put $b=-b_{1}^{2}$.
If $\left( u_{n}(a,b) \right)_{n \geq 0}$ is a non-degenerate binary recurrence sequence,
then $u_{n}(a,b) \in c\rS$ for $n>8$ and $|c|=1$ or $|c|$ prime if
and only if $(n,a,b,c)=(9,1,-4,-19)$, $(23,1,-4,-1747)$, $(11,4,-25,3719)$ or $(13,3,-4,181)$.
\end{conjecture}

In the special, but interesting, case when $b_{1}=1$, a sharper result appears
to hold.

\begin{conjecture}
\label{conj:new2}
For $a>2$, $|c|=1$ or $|c|$ prime, and $n>4$,
$u_{n}(a,-1) \in c\rS$ if and only if $(n,a,c) = (6,3,1)$ or $(7,6,239)$.
\end{conjecture}

What motivates, and underpins, our conjectures is that, for $b=-b_{1}^{2}$,
$u_{n}(a,b)$ can be factored as elements of other binary recurrence sequences
and this factorisation prevents $u_{n}(a,b)$ from being a square or a prime times
a square for $n$ sufficiently large.
More precisely, for $b=-b_{1}^{2}$, we define the sequences $\left( t_{n}(a,b) \right)_{n \geq 0}$,
$\left( v_{n}(a,b) \right)_{n \geq 0}$ and $\left( w_{n}(a,b) \right)_{n \geq 0}$
by $t_{0}=1$, $t_{1}=a-b_{1}$, $v_{0}=2$, $v_{1}=a$,
$w_{0}=1$, $w_{1}=a+b_{1}$ and all subject to the same recurrence relation as 
$u_{n}(a,b)$.
Then
\begin{equation}
\label{eq:u=t*w}
u_{2n+1}(a,b)=t_{n}(a,b)w_{n}(a,b)
\end{equation}
and
\[
u_{2n}(a,b)=u_{n}(a,b)v_{n}(a,b).
\]

The factorisation of $u_{2n}(a,b)$ holds for all binary recurrence sequences with
$u_{0}=0$ and $u_{1}=1$. It is the factorisation of $u_{2n+1}(a,b)$ as the product
of two elements from recurrence sequences that is significant here. It also
appears to be novel. We know of no previous occurrences of this relationship in the literature.
It only holds when $b=-b_{1}^{2}$. Hence our restriction to this case above.

For any $n$, $\gcd \left( t_{n}(a,b), w_{n}(a,b) \right)$ and
$\gcd \left( u_{n}(a,b), v_{n}(a,b) \right)$ are both at most $2$ (see Lemma~\ref{lem:gcd}
below). Hence the problem of finding near-squares in our $\left( u_{n} \right)_{n \geq 0}$
comes down to finding elements in $\rS$ and $2\rS$ in these four sequences.

Using PARI/GP \cite{Pari}, we also checked these conjectures by searching for
near-squares among the $u_{n} \left( a, -b_{1}^{2} \right)$ for non-degenerate
sequences with $1 \leq a, b_{1} \leq 2000$, $\gcd \left( a, b_{1} \right)=1$,
$n \leq 60$ and $\left| u_{n} \right|<10^{50}$. No further examples were found.

Concerning our conjectures above, we have obtained results when $b=-1$ and $c>3$
a prime (note that $c=2$ and $c=3$ were treated by Mignotte and Peth\H{o} \cite{MP}).
To simplify our notation, instead of 
writing $t_{n}(a,-1)$, $u_{n}(a,-1)$, $v_{n}(a,-1)$ and $w_{n}(a,-1)$, we will
simply write $t_{n}$, $u_{n}$, $v_{n}$ and $w_{n}$, except if it is necessary to stress the 
dependence on the parameter $a$; in such cases, we will write $t_{n}(a)$, $u_{n}(a)$,
$v_{n}(a)$ or $w_{n}(a)$.

In Appendix~\ref{app:ljunggren}, we provide a translation of the statements of
the theorems in \cite{Ljun54}, with Theorem~IV there corrected. On the arXiv
\cite{Ljun-arxiv}, we also provide the complete text of this paper in its original
German and our English translation of it. As this paper is difficult to obtain,
we hope that this is also of interest and benefit to researchers in this area.

We hope that this work motivates others to generalise our results and investigate
our conjectures, as they are certainly of considerable interest for the study of
the arithmetic properties of binary recurrence sequences.

\section*{Acknowledgements}

Our work benefitted from discussions with Gary Walsh,
so we gratefully thank him for his time and suggestions.
We also thank Andrew Granville for his questions about possible connections
between our factorisation, \eqref{eq:u=t*w}, and Aurifeuillean factorisations.
His questions led us to investigating and understanding this connection.
Lastly, we thank the referees for their careful reading and suggestions, which considerably improved
the presentation of the paper.

\section{Our results}
\label{sec:results}

Recall that, as mentioned above, $c\rS$ with $c=1$ ,$2$, $3$, as well as $c=6$,
was previously treated by \cite{MP}; see Proposition~\ref{prop:mp}(a) below.
Therefore, from Section~\ref{sec:u_2n} onward, we will assume that the prime
$p$ is at least $5$.

The last part of our main result (Theorem~\ref{thm:at most one solution}) is
based on the following conjecture of Togb\'{e}, Voutier and Walsh.

\begin{conjecture}[Conjecture 1.1 of \cite{TogVouWal2005}]
\label{conj:TogVouWalsh}
If $t>1$ is an integer, then the only positive integer solution of
$(t+1)X^{4}-tY^{2}=1$ is $(X,Y)=(1,1)$, except if $t=m^{2}+m$ $(m>0)$, 
in which case there is also the solution $(X,Y)=\left( 2m+1,4m^{2}+4m+3 \right)$.
\end{conjecture}

\subsection{Our main result}
\label{subsec:results}

Now we state the main result of this paper.

\begin{theorem}
\label{thm:at most one solution}
The relation $u_{N}(a)\in p\rS$ with $a\geq 3$, $N \geq 5$ and prime $p$
holds only if all three of the following conditions \emph{(i)--(iii)} hold:

\emph{ (i)} $p \geq 5$,

\emph{ (ii)} $a\equiv 2 \pmod{4}$,

\emph{(iii)} either $N=7$, in which case the only instance of $u_{7}(a)\in p\rS$
is $(a,p)=(6,239)$, or $N \geq 19$ is a prime with $N \equiv 3 \pmod{4}$.

Moreover, for fixed $a \equiv 2 \pmod{4}$, there exists at most one pair $(N,p)$
such that $u_{N}(a)\in p\rS$.

Finally, if Conjecture~\ref{conj:TogVouWalsh} is true, then 
$u_{N}(a)\in p\rS$ with $N\geq 5$, $a\geq 3$ and prime $p$ holds only when $(a,N,p)=(6,7,239)$.
%
%If $a>2$ is even and $n>1$, then $t_{n}(a)$ can be a square only when $a \equiv 2 \pmod{4}$,
%$n$ is odd and $2n+1$ is prime $\geq 19$.
%For such $a$, there is at most one $n>1$ with $t_{n}(a)$ a square.
\end{theorem}

\begin{remark-nonum}
From the last statement of our theorem, it follows that a proof of Conjecture~\ref{conj:TogVouWalsh}
would also prove our Conjecture~\ref{conj:new2}.
\end{remark-nonum}

\begin{proof}
First we show that both $u_{N}(a)\in 2\rS$ and $u_{N}(a)\in 3\rS$ are impossible
with $a \geq 3$ and $N\geq 5$. Indeed, if $a\geq 4$ and $N\geq 5$, then both
$u_{N}(a)\in 2\rS$ and $u_{N}(a)\in 3\rS$ are 
impossible by a result of Mignotte-Peth\H{o} \cite{MP} (see Proposition~\ref{prop:mp}(a)).
Next, let $a=3$ and $N\geq 5$.
If $u_{N}(3)\in 2\rS$, then $N=6$, in view of Lemma~\ref{lem:RibMcDan}(b),
but $u_{6}(3)=12^{2}$, hence $u_{N}(3)\not\in 2\rS$. Also, $u_{N}(3)\not\in 3\rS$
by Lemma~\ref{lem:u_n(3)}. Thus we have proved the necessity of condition~(i).

For the remainder of the proof, we assume that $u_{N}(a)\in p\rS$
with $a \geq 3$, $N \geq 5$ and prime $p \geq 5$.

By Proposition~\ref{prop:u_2n in pS} and Proposition~\ref{prop:u_2n+1 not square with odd a},
we have $N$ is odd, say $N=2n+1$ with $n \geq 2$, and $a$ is even.
By Lemma~\ref{lem:t_n or w_n must be square}, it follows that either
$w_{n}(a)\in\rS$ or $t_{n}(a)\in\rS$ (cf.~\eqref{eq:u=t*w}).
By Lemma~\ref{lem:w_n in S}, we cannot have $w_{n}(a) \in \rS$. Hence
$t_{n}(a)\in\rS$.

By Lemma~\ref{lem:t_n square,general a}, $a \equiv 2 \pmod{4}$,
$n$ is odd and $2n+1$ is equal either to $7$ (in which case $a=6$) or to a prime $\geq 19$.
Thus we have proved the necessity of conditions~(ii) and (iii).

Moreover, for such a fixed $a$, Lemma~\ref{lem:t_n square,general a}
also asserts that there is at most one 
$n>1$ with $t_{n}(a)$ a square. Hence for any fixed $a$, there exists at most one
pair $(N,p)$ such that $u_{N}(a)\in p\rS$.

Finally, assume the truth of Conjecture~\ref{conj:TogVouWalsh}. According
to what we have already proved, $a \equiv 2 \pmod{4}$ and $t_{n}(a) \in \rS$, where
$N=2n+1$. Therefore, as in the beginning of the proof of 
Lemma~\ref{lem:t_n square,general a}, let us put $t_{n}(a)=T^{2}$ and $w_{n}(a)=W$,
so that
$((a+2)/4)T^{4}-((a-2)/4)W^{2}=1$ (cf. \eqref{eq:(a+2)/4*T^4-(a-2)/4*W^2=1}).
This is an equation which falls under the scope of Conjecture~\ref{conj:TogVouWalsh}.
According to that conjecture, no positive integer solution $(T,W)\neq (1,1)$ exists except if
$a-2=4 \left( m^{2}+m \right)$ for some positive integer $m$, in which case we also have the
solution
$(T,W)=\left( 2m+1,4m^{2}+4m+3 \right)$. This implies that 
$\left( t_{n}(a),w_{n}(a) \right)=((2m+1)^{2},4m^{2}+4m+3)$.
In particular, $w_{n}(a)-t_{n}(a)=2$ which is impossible because using
\eqref{eq:tn-closed}, \eqref{eq:un-closed} and \eqref{eq:wn-closed} we find that
$w_{n}(a)-t_{n}(a)=2u_{n}(a)>2$ for $n \geq 2$. 
Thus, Conjecture~\ref{conj:TogVouWalsh} implies that 
$u_{N}(a)\in p\rS$ is impossible under our conditions.
\end{proof}

The proof of Theorem~\ref{thm:at most one solution} is based on several intermediate
results established in this paper. We believe that the ``tree'' in
Appendix~\ref{app:interdependence} may be helpful to the reader.
It shows the interdependence of the various results proved in this paper.

Our study of $u_{n}(a)\in p\rS$ focuses only on subscripts $n>4$.

For $n=2$, we have $u_{2}(a)=a$, so near-squares among $u_{2}(a)$ are not very
interesting.

We have $u_{3}(a)=a^{2}-1$, so that $u_{3}(a)\in p\rS$ leads to an equation $a^{2}-py^{2}=1$ 
which, for every prime $p$ has infinitely many solutions $(a,y)$.

Finally, $u_{4}(a)=a^{3}-2a$ and it is easy to find pairs $(a,p)$ with prime $p\geq 5$,
such that $a^{3}-2a\in p\rS$, for example, $(a,p)=(10,5), (58,29), (8,31)$ and $(9,79)$.
Also if $a=2p$ and $2p^{2}-1$ is a square, then $u_{4}(a) \in pS$ and it is
conjectured that there are infinitely many primes, $p$, such that $2p^{2}-1$ is
a square. There are also many other conjecturally infinite families of examples
like this.

\subsection{Our Factorisation of $u_{2n+1} \left( a,-b_{1}^{2} \right)$}
\label{subsec:factor}

Aurifeuillian factorisations are factorisations of $x^{n}-y^{n}$ or (what is
essentially the same) the $n$-th cyclotomic polynomial, $\Phi_{n}(x,y)$,
for special choices of $x$ and $y$.
They arise from being able to write $x^{n}-y^{n}$ or $\Phi_{n}(x,y)$
as a difference of two squares for these special values of $x$ and $y$.

A well-known example is
$2^{4n+2}+1=\left( 2^{2n+1}-2^{n+1}+1 \right) \left( 2^{2n+1}+2^{n+1}+1 \right)$.
This arises from setting $x=2^{2n+1}$ and $y=1$ in
$\Phi_{4}(x,y)=x^{2}+y^{2}=(x+y)^{2}-2xy$.

These expressions, and the resulting factorisations, go back to Gauss \cite[Art. 356--357]{Gauss},
who found that $4\Phi_{n}(x)=A_{n}(x)^{2}-(-1)^{(n-1)/2}nB_{n}(x)^{2}$ for
odd square-free $n>1$ and $A_{n}(x), B_{n}(x) \in \bbZ[x]$. They have been studied
extensively since then. We refer the reader to \cite{Brent1, GP, Sch, St} as
some of the significant papers from the past 60 years. They also provide a more
extensive history of the subject.

Most of the Aurifeullian factorisations are for rational integer values of $x$
and $y$, although a few Aurifeullian factorisations are known for other numbers,
like
\[
L_{10n+5}=L_{2n+1} \left( 5F_{2n+1}^{2}-5F_{2n+1}
              +1 \right) \left( 5F_{2n+1}^{2}+5F_{2n+1}+1 \right),
\]
where $F_{n}$ is the $n$-th Fibonacci number and $L_{n}$ is the $n$-th Lucas
number -- here $x,y=\left( 1 \pm \sqrt{5} \right)/2$.

The factorisation in equation~\eqref{eq:u=t*w} looks suspiciously like an
Aurifeullian factorisation. Our investigations following Andrew Granville's
questions led to us understanding the Aurifeullian connection.

For $n \geq 0$, we can write
\begin{equation}
\label{eq:Aurifeullian}
\frac{x^{2n+1}-y^{2n+1}}{x-y} = \left( \frac{x^{n+1}-y^{n+1}}{x-y} \right)^{2}
-xy \left( \frac{x^{n}-y^{n}}{x-y} \right)^{2}.
\end{equation}

This relationship is easily proven by routine algebraic manipulation. If $xy$
is a square, then the right-hand side is a difference of two squares. Furthermore,
the two terms on the right-hand side are symmetric functions in two
variables, so if we let $x$ and $y$ be algebraic integers of degree $2$ that are
algebraic conjugates, then both terms are rational integers leading to a
difference of two squares factorisation of the left-hand side.

Letting $x=\alpha$ and $y=\beta$, which are algebraic conjugates, then
$xy=b_{1}^{2}$ and we obtain an Aurifeullian factorisation of 
$u_{2n+1} \left( a, -b_{1}^{2} \right)$.
We show below in Lemma~\ref{lem:u=t*w} that this Aurifeullian factorisation is
our factorisation in \eqref{eq:u=t*w}.

If $2n+1$ is prime, then this factorisation also follows from Theorem~1 of \cite{Sch}
with $m=1$.

As well as its importance for our work here, our factorisation in \eqref{eq:u=t*w},
and those in the next subsection, also shows that when not restricting to rational
integer values of $x$ and $y$, there are an abundance of other Aurifeullian factorisations.

\subsection{Further Factorisations}
\label{subsec:factor2}

We found more factorisations as well. We can write
\[
u_{m(2n+1)} \left( a, -mb_{1}^{2} \right)
= u_{2n+1}\left( a, -mb_{1}^{2} \right) f_{1,m,n} \left( a, -mb_{1}^{2} \right)
f_{2,m,n} \left( a, -mb_{1}^{2} \right),
\]
where $f_{1,m,n}, f_{2,m,n} \in \bbZ[x,y]$ and $m \equiv 1 \pmod{4}$ is prime.

After realising the relevance of Schinzel's work \cite{Sch} for such
factorisations, we found that these factorisations also come from Theorem~1
there with $n=m$ in the notation of \cite{Sch}. The condition that $m \equiv 1 \pmod{4}$
arises from equation~(1) in Theorem~1 of \cite{Sch}. In order for the right-hand
side to be a difference of two squares, we need the Jacobi symbol, $(-1|m)$, to
equal $1$. As is well-known, for $m$ an odd prime, this only occurs for $m \equiv 1 \pmod{4}$.

\section{Preliminaries}
\label{sec:prelim}

We collect in this section some results about our sequences that we will need
throughout our work.

Using the theory of recurrence sequences (see, for example, Chapter~C of \cite{ShTi})
and the recurrence relation satisfied by the four sequences, we find that the
companion polynomial of the four sequences is $x^{2}-ax+b_{1}^{2}$.
Its roots are $\alpha$ and $\beta$ given in \eqref{eq:seq-defn}
(so $\Delta=a^{2}-4b_{1}^{2}$).
We can write
\begin{align}
\label{eq:tn-closed}
t_{n}\left( a, -b_{1}^{2} \right) &= 
\frac{(\alpha-b_{1})\alpha^{n}-(\beta-b_{1})\beta^{n}}{\sqrt{\Delta}}, \\
\label{eq:un-closed}
u_{n}\left( a, -b_{1}^{2} \right) &= \frac{\alpha^{n} - \beta^{n}}{\alpha-\beta}
= \frac{\alpha^{n} - \beta^{n}}{\sqrt{\Delta}}, \\
\label{eq:vn-closed}
v_{n}\left( a, -b_{1}^{2} \right) &= \alpha^{n}+\beta^{n} \text{ and} \\
\label{eq:wn-closed}
w_{n}\left( a, -b_{1}^{2} \right) &= 
\frac{(\alpha+b_{1})\alpha^{n}-(\beta+b_{1})\beta^{n}}{\sqrt{\Delta}}.
\end{align}

\begin{lemma}
\label{lem:u=t*w}
For $n \geq 0$, \eqref{eq:u=t*w} holds. That is,
\[
u_{2n+1}\left( a, -b_{1}^{2} \right)=t_{n}\left( a, -b_{1}^{2} \right)w_{n}
\left( a, -b_{1}^{2} \right).
\]

Furthermore, this is the Aurifeullian factorisation of $u_{2n+1}$ from 
Subsection~\ref{subsec:factor}.
\end{lemma}

\begin{proof}
This relationship follows by expanding the right-hand side using \eqref{eq:tn-closed}
and \eqref{eq:wn-closed}, and then simplifying to find that it equals \eqref{eq:un-closed}
for $u_{2n+1}$.

Substituting $x=\alpha$ and $y=\beta$ into \eqref{eq:Aurifeullian} and factoring, we obtain
\[
u_{2n+1}\left( a, -b_{1}^{2} \right) =
\left( \frac{\alpha^{n+1}-\beta^{n+1}}{\alpha-\beta} - 
\sqrt{\alpha\beta} \frac{\alpha^{n}-\beta^{n}}{\alpha-\beta} \right)
\left( \frac{\alpha^{n+1}-\beta^{n+1}}{\alpha-\beta} 
+ \sqrt{\alpha\beta} \frac{\alpha^{n}-\beta^{n}}{\alpha-\beta} \right).
\]

Choosing the square roots such that $\sqrt{\alpha} \sqrt{\beta}=b_{1}$, this becomes
\[
u_{2n+1}\left( a, -b_{1}^{2} \right) =
\left( \frac{\alpha^{n+1/2}+\beta^{n+1/2}}{\sqrt{\alpha}+\sqrt{\beta}} \right)
\left( \frac{\alpha^{n+1/2}-\beta^{n+1/2}}{\sqrt{\alpha}-\sqrt{\beta}} \right).
\]

A routine manipulation using \eqref{eq:tn-closed} and \eqref{eq:wn-closed} shows that the two terms 
on the right-hand side are $t_{n}\left( a, -b_{1}^{2} \right)$ and $w_{n}\left( a, -b_{1}^{2} \right)$, 
respectively.
\end{proof}

\begin{lemma}
\label{lem:seq-relationships}
\begin{align}
\label{eq:un2-vn2-relationship}
v_{n}^{2}\left( a, -b_{1}^{2} \right) - 
\left( a^{2}-4b_{1}^{2} \right) u_{n}^{2}\left( a, -b_{1}^{2} \right)
&= 4b_{1}^{2n},\\
\label{eq:tn2-wn2-relationship}
\left( a+2b_{1} \right)t_{n}^{2}\left( a, -b_{1}^{2} \right)
- \left( a-2b_{1} \right)w_{n}^{2}\left( a, -b_{1}^{2} \right)
&= 4b_{1}^{2n+1}.
\end{align}
\end{lemma}

\begin{proof}
These relationships follow by expanding their left-hand sides using \eqref{eq:un-closed}
and \eqref{eq:vn-closed} for \eqref{eq:un2-vn2-relationship}, as well as \eqref{eq:tn-closed}
and \eqref{eq:wn-closed} for \eqref{eq:tn2-wn2-relationship}, and then
simplifying the resulting expressions.
\end{proof}

\begin{lemma}
\label{lem:gcd}
{\rm (a) } If $\gcd(a,b)=1$, then
\begin{equation}
\label{eq:gcd u_{n},v_{n}}
\gcd \left( u_{n}(a,b), v_{n}(a,b) \right) =
\begin{cases}
  2 & \text{if $b$ is odd and either (1) $a$ is odd and $3\mid n$} \\
    & \text{or (2) both $a$ and $n$ are even,} \\
  1 & \text{otherwise.}
\end{cases} 
\end{equation}

{\rm (b)} If $\gcd \left( a, b_{1} \right)=1$, then
\begin{equation}
\label{eq:gcd(t,w)}
\gcd \left( t_{n} \left( a, -b_{1}^{2} \right), w_{n} \left( a, -b_{1}^{2} \right) \right)
=\begin{cases}
2 & \text{if $a$ and $b_{1}$ are both odd and $n \equiv 1 \pmod{3}$,} \\
1 & \text{otherwise.}
\end{cases}
\end{equation}
\end{lemma}

\begin{proof}
(a) This is Theorem~0 of \cite{McD} as well as the statement after the Main Theorem
there. But it goes back to Lucas, as the author writes there.

(b) Since $\gcd \left( a, b_{1} \right)=1$, we find that 
$\gcd \left( t_{n}\left( a, -b_{1}^{2} \right), b_{1} \right)
=\gcd \left( w_{n}\left( a, -b_{1}^{2} \right), b_{1} \right)=1$. Hence, from
\eqref{eq:tn2-wn2-relationship}, we have 
$\gcd \left( t_{n}\left( a, -b_{1}^{2} \right), w_{n}\left( a, -b_{1}^{2} \right) \right)|2$.

Also both $t_{n} \left( a, -b_{1}^{2} \right)$ and $w_{n} \left( a, -b_{1}^{2} \right)$
are always odd when either $a$ or $b$ is even. So if\\
$\gcd \left( t_{n}\left( a, -b_{1}^{2} \right), w_{n}\left( a, -b_{1}^{2} \right) \right)=2$,
then both $a$ and $b_{1}$ must be odd.

By induction, if $a$ and $b$ are both odd, then $t_{n}\left( a, -b_{1}^{2} \right)$
and $w_{n}\left( a, -b_{1}^{2} \right)$ are both odd when $n \equiv 0,2 \pmod{3}$
and both even when $n \equiv 1 \pmod{3}$.
\end{proof}

We shall require the following observations several times in what follows.

\begin{remark}
\label{rem:increasing}
For $a \geq b_{1}+2$, the element of index $1$ in each of these sequences is larger
than the element of index $0$. If, in addition, $a>b_{1}^{2}+1$, then the recurrence
relation and induction show that the four sequences are strictly increasing as $n$
increases. So for $b_{1}=1$ and $a \geq 3$, the four sequences are strictly
increasing as $n$ increases.

From this, it also follows that for $b_{1}=1$ and any fixed $n \geq 2$, the four
sequences (here indexed by $a$) are strictly increasing as $a \geq 3$ increases.
\end{remark}

\section{Results from the literature and some consequences}
\label{sec:results on Lucas seqs}

\subsection{Results on Lucas sequences}
\label{subsec:results on Lucas seqs}

Some of the results obtained in \cite{RibMcDan} are included in the following 
lemma.

\begin{lemma}
\label{lem:RibMcDan}
Let $a,b$ be relatively prime odd integers with $a^{2}+4b>0$.

\emph{(a)}\: If $n>0$ and $u_{n}(a,b) \in \rS$, then $n\in\{1,2,3,6,12\}$.

\emph{(b)}\: If $n>0$ and $u_{n}(a,b) \in 2\rS$, then $n\in\{3,6\}$.

\emph{(c)}\: If $n>0$ and $v_{n}(a,b) \in\rS$, then $n\in\{1,3,5\}$.\\
Furthermore,
if $v_{3}(a,b) \in \rS$, then $-b \equiv 3 \pmod{4}$ and if $v_{5}(a,b) \in \rS$,
then $a \equiv -b \equiv 5 \pmod{8}$.

\emph{(d)}\: If $n>0$ and $v_{n}(a,b) \in 2\rS$, then $n\in\{3,6\}$.
\end{lemma}

\begin{proof}
Note that $U_{n}(a,b)$ and $V_{n}(a,b)$ in \cite{RibMcDan} are $u_{n}(a,-b)$
and $v_{n}(a,-b)$, respectively, in our notation.
\\
(a)\: This is \cite[Theorem~3(a)]{RibMcDan}.
\\
(b)\: This is \cite[Theorem~4(a)]{RibMcDan}.
\\
(c)\: This is \cite[Theorem~1]{RibMcDan}.
\\
(d)\: This is \cite[Theorem~2(a)]{RibMcDan}.
\end{proof}

\subsection{Results on Diophantine equations}
\label{subsec:results on Diophantine eqns}

We collect here the statements of some known results on the
integer solutions of Diophantine equations of the form $Ax^{4}-By^{2}=C$
for $C=\pm 1, \pm 2, \pm 4$ that we will use to prove our results. We preserve
the notation of the original papers.

\begin{lemma}
\label{lem:Cohn66}
Let $d$ be a positive integer such that $X^{2}-dY^{2}=-4$ has an integer solution
with $XY$ odd.
If $d=5$, then the only integer solutions of $y^{2}-dx^{4}=4$ with $x \geq 0$
are $x=0$, $1$, and $12$.
If $d \neq 5$, then the only integer solutions of $y^{2}-dx^{4}=4$ with $x \geq 0$
are $x=0$ and possibly one other.
\end{lemma}

\begin{proof}
This is the ``Equation~6'' result stated on page~163 of \cite{Cohn66}.
\end{proof}

\begin{lemma}
\label{lem:Cohn67}
Let $d$ be a positive integer such that $X^{2}-dY^{2}=4$ has odd solutions, but
$X^{2}-dY^{2}=-4$ has no odd solutions. Let $\left( a_{1}, b_{1} \right)$ be the
minimal solution in odd positive integers of $X^{2}-dY^{2}=4$.
Apart from $x=0$, the equation $y^{2}-dx^{4}=4$ has at most two solutions with $x \geq 0$.
They come from $x=b_{1}$, $x=a_{1}b_{1}$ or $x=b_{1} \left( a_{1}^{2}-1 \right)$.
\end{lemma}

\begin{proof}
First note that $y^{2}-dx^{4}=4$ is covered by the ``Equation~4'' result stated
on page~69 of \cite{Cohn67}.

Let $(X,Y)=\left( a_{1}, b_{1} \right)$ be the smallest odd solution of
$X^{2}-dY^{2}=4$. Note that to avoid confusion with our notation here, we have
used $a_{1}$ in place of Cohn's $a$ and $b_{1}$ in place of his $b$ (both
defined after equation~(2) of his paper \cite{Cohn67}).

For the same reason, we denote by $\left( U_{n} \right)_{n \geq 0}$ the sequence
that Cohn denotes by $\left( u_{n} \right)_{n \geq 0}$ on page~62 of in \cite{Cohn67}.
So $U_{0}=0$, $U_{1}=b_{1}$, $U_{2}=a_{1}U_{1}-U_{0}=a_{1}b_{1}$ and
$U_{3}=a_{1}U_{2}-U_{1}=a_{1}^{2}b_{1}-b_{1}=b_{1} \left( a_{1}^{2}-1 \right)$.

As Cohn states after his ``Equation~4'' result, the possible solutions with
$x$ a positive integer come from at most two of $x^{2}=U_{1}=b_{1}$,
$x^{2}=U_{2}=a_{1}b_{1}$ and $x^{2}=U_{3}=b_{1} \left( a_{1}^{2}-1 \right)$.

In his Theorem~3 of \cite{Cohn67}, Cohn states that $x^{2}=U_{1}$ is only possible
if $b_{1}$ is a perfect square, $x^{2}=U_{2}$ is only possible if both $a_{1}$
and $b_{1}$ are perfect squares and that $x^{2}=U_{3}$ is only possible if $b_{1}$
is $3$ times a square.
\end{proof}

Let $A$ and $B$ be odd positive integers and suppose that $Az_{1}^{2}-Bz_{2}^{2}=4$
has solutions in odd positive integers $z_{1}$ and $z_{2}$. Furthermore, let
$\left( a_{1},b_{1} \right)$ be the least solution of $Az_{1}^{2}-Bz_{2}^{2}=4$
in odd positive integers.

\begin{lemma}
\label{lem:L67}
The diophantine equation $Ax^{4}-By^{2}=4$ has at most two solutions in positive
integers $x$ and $y$. If $a_{1}=h^{2}$ and $Aa_{1}^{2}-3=k^{2}$, there are two
solutions, namely $x=h$ and $x=hk$. If $a_{1}=h^{2}$ and $Aa_{1}^{2}-3 \neq k^{2}$,
there is one solution, $x=h$. If $a_{1}=5h^{2}$ and 
$A^{2}a_{1}^{4}-5Aa_{1}^{2}+5=5k^{2}$,
there is one solution, $x=5hk$.
\end{lemma}

\begin{proof}
This is Theorem~1 of \cite{Ljun67}.
\end{proof}

Let $D$ be a positive nonsquare integer and $\varepsilon_{D}=T_{1}+U_{1}\sqrt{D}$
be the minimal unit greater than $1$ of norm $1$ in
$\bbZ \left[ \sqrt{D} \right]$. For $k \geq 1$, write $\varepsilon_{D}^{k}
=T_{k}+U_{k}\sqrt{D}$.

\begin{lemma}
\label{lem:TVW2005}
The equation $X^{2}-DY^{4}=1$ has at most two solutions in positive integers,
$X$ and $Y$. If two such solutions,
$\left( X_{1}, Y_{1} \right)$ and $\left( X_{2}, Y_{2} \right)$ exist with
$Y_{1}<Y_{2}$, then $Y_{1}^{2}=U_{1}$ and $Y_{2}^{2}=U_{2}$, except if
$D=1785$ or $D=16 \cdot 1785$, in which cases $Y_{1}^{2}=U_{1}$ and
$Y_{2}^{2}=U_{4}$.
\end{lemma}

\begin{proof}
This is Theorem~1.1(i) in \cite{TogVouWal2005}.
\end{proof}

\subsection{Some consequences}
\label{subsec:consequences}

\emph{
Throughout the remainder of the paper, we will deal with the case of $b=-1$.
For this reason, we will simplify the notation by writing $t_{n}(a)$, $u_{n}(a)$,
$v_{n}(a)$ and $w_{n}(a)$ instead of $t_{n}(a,-1)$, $u_{n}(a,-1)$, $v_{n}(a,-1)$
and $w_{n}(a,-1)$, respectively.}
 
Using Lemma~\ref{lem:RibMcDan}, we prove the following result.

\begin{lemma}
\label{lem:u_n(3)}
$u_{n}(3) \not\in 3\rS\cup 6\rS$ for every $n\geq 3$.
\end{lemma}

\begin{proof}
For simplicity, we will write $u_{n}$ instead of $u_{n}(3)$ and $v_{n}$ instead
of $v_{n}(3)$. Assume that $u_{n} \in 3\rS\cup 6\rS$. So $3 \mid u_{n}$. We
have $u_{2}=a=3$ here. Furthermore, $\left( u_{k}(a,b) \right)_{k \geq 0}$ is a
strong divisibility sequence, so
$u_{\gcd(k_{1},k_{2})}(a,b)=\gcd \left( u_{k_{1}}(a,b), u_{k_{2}}(a,b) \right)$.
It follows that $n$ is even, and we put $n=2m$, so that $u_{n}=u_{2m}=u_{m}v_{m}$.
Recall from Lemma~\ref{lem:gcd} that
$\gcd \left( u_{m},v_{m} \right)=1$ or $2$, according as $3 \nmid m$ or $3 \mid m$.
Also, $m \geq 2$ because $n \geq 3$.

First we consider the case $u_{n} \in 3\rS$. If $3 \nmid m$, then, from
$u_{n}=u_{m}v_{m} \in 3\rS$ and $\gcd \left( u_{m},v_{m} \right)=1$,
it follows that either $u_{m}$ or $v_{m}$ is a square. By parts~(a) and (c)
of Lemma~\ref{lem:RibMcDan} and $3 \nmid m$, it follows that $m\in\{2,5\}$. We
complete these remaining cases at the end of the proof.

If $3\mid m$, then $\gcd \left( u_{m},v_{m} \right)=2$. This, combined with
$u_{m}v_{m} \in 3\rS$ shows that either $u_{m}$ or $v_{m}$ belongs to $2\rS$.
From parts~(b)
and (d) of Lemma~\ref{lem:RibMcDan}, it follows that $m \in \{3,6\}$.
Once again, we complete these remaining cases at the end of the proof.

Next, consider the case $u_{n} \in 6\rS$.
If $3 \nmid m$, then, from $u_{m}v_{m}\in 6\rS$
and $\gcd \left( u_{m},v_{m} \right)=1$, we have to consider the following four cases.

\noindent
(I) $u_{m} \in 2\rS$ and $v_{m} \in 3\rS$. The first inclusion is impossible by
Lemma~\ref{lem:RibMcDan}(b).

\noindent
(II) $u_{m} \in 3\rS$ and $v_{m} \in 2\rS$. The second inclusion is impossible by
Lemma~\ref{lem:RibMcDan}(d).

\noindent
(III) $u_{m} \in 6\rS$ and $v_{m} \in\rS$. The second inclusion implies $m=5$ by
Lemma~\ref{lem:RibMcDan}(c).

\noindent
(IV) $u_{m} \in \rS$ and $v_{m} \in 6\rS$. The first inclusion implies $m=2$ by
Lemma~\ref{lem:RibMcDan}(a).
If $3\mid m$, then $\gcd \left( u_{m},v_{m} \right)=2$. If we combine this with
$u_{m}v_{m}\in 6\rS$, then we see that one of the following four cases must hold.

\noindent
(I) $u_{m}\in\rS$ and $v_{m}\in 6\rS$. The first inclusion implies $m\in\{3,6,12\}$ by
Lemma~\ref{lem:RibMcDan}(a).

\noindent
(II) $u_{m}\in 6\rS$ and $v_{m}\in\rS$. The second inclusion implies $m=3$ by
Lemma~\ref{lem:RibMcDan}(c).

\noindent
(III) $u_{m}\in 3\rS$ and $v_{m}\in 2\rS$. The second inclusion implies $m\in\{3,6\}$ by
Lemma~\ref{lem:RibMcDan}(d).

\noindent
(IV) $u_{m}\in 2\rS$ and $v_{m}\in 3\rS$. The first inclusion implies $m\in\{3,6\}$ by
Lemma~\ref{lem:RibMcDan}(b).

In all cases above, $2\leq m\leq 12$ which implies that $4\leq n\leq 24$.
A straightforward computation shows that $u_{n} \not\in 3\rS\cup 6\rS$ 
for every $n=4,\ldots,24$.
\end{proof}

The following result is used very often in our arguments in 
Subsections~\ref{subsec:u_2n a odd and n=0 mod 3} 
through~\ref{subsec:a odd and n=1 mod 3}. Part~(a) is the main result of
Mignotte-Peth\H{o} \cite{MP}.

\begin{proposition}%[Mignotte-Peth\H{o}]
\label{prop:mp}
\emph{(a)} Let $a\geq 4$ and $N\geq 4$. Then $u_{N}(a)\in\rS$ is possible only 
when $a=338$ and $N=4$. For $c\in\{2,3,6\}$, $u_{N}(a)\in c\rS$ is impossible.

\emph{(b)} Let $a=3$. Then $u_{N}(3)\in\rS$ only when $N=6$; $u_{N}(3)\in 2\rS$ only 
when $N=3$. For $c\in\{3,6\}$, $u_{N}(3)\in c\rS$ is impossible for every $N\geq 3$

\emph{(c)} Let $N=3$. Then $u_{3}(a)\in\rS$ is impossible for every $a\geq 3$.
For each $c \in \{2,3,6\}$, there are infinitely many $a \geq 4$ such that
$u_{3}(a)\in c\rS$ holds.
\end{proposition}

\begin{proof}
(a) This is the main theorem of \cite{MP}.

(b) Lemma~\ref{lem:u_n(3)} implies that $u_{N}(3)\not\in 3\rS\cup 6\rS$ for
every $N\geq 3$. Suppose that $u_{N}(3)\in \rS\cup 2\rS$. 
Then, by Lemma~\ref{lem:RibMcDan}(a) and (b), it follows that 
$N\in\{3,6,12\}$. 
We have $u_{3}(3)=2^{3}$. Therefore, $u_{3}(3)\in 2\rS$; $u_{6}(3)=2^{4} \cdot 3^{2}$.
Hence $u_{6}(3)\in\rS$; 
$u_{12}(3)=2^{5} \cdot 3^{2} \cdot 161$, and so $u_{12}(3)\not\in \rS\cup 2\rS$. 

(c) The proof is immediate from the theory of Pell equations since $u_{3}(a)=a^{2}-1$.
\end{proof}

\section{$u_{2n}(a)\in p\rS$ for $a \geq 3$, $n\geq 3$ and $p \geq 5$}
\label{sec:u_2n}

This section is devoted to the proof of the following result.

\begin{proposition}
\label{prop:u_2n in pS}
$u_{N}(a) \in p\rS$ is impossible with $a \geq 3$, $N=2n \geq 6$ and prime $p \geq 5$.
\end{proposition}

We will prove this proposition by applying 
Lemmas~\ref{lem:u_2n not in pS, a odd, n=0(mod3)},
\ref{lem:u_2n a and n even} and \ref{lem:u_2n in pS => u_m in pS}, which 
are stated and proved in Subsections~\ref{subsec:u_2n a odd and n=0 mod 3},
\ref{subsec:u_2n a,n even}, and~\ref{subsec:u_2n gcd(u_n,v_n)=1}, respectively.

\emph{
In the remainder of this section, we will always assume that $a\geq 3$, $n\geq 3$
and $p$ is a prime $\geq 5$. No further mention of these assumptions will be
made in this section.}

Recall that $u_{2n}(a)=u_{n}(a)v_{n}(a)$ and, from Lemma~\ref{lem:gcd}(a) that
$\gcd \left( u_{n}(a),v_{n}(a) \right)=1,2$.

The two cases when $\gcd \left( u_{n}(a),v_{n}(a) \right)=2$
are treated in separate subsections 
(Subsections~\ref{subsec:u_2n a odd and n=0 mod 3} and \ref{subsec:u_2n a,n even}
respectively) and the two cases when $\gcd \left( u_{n}(a),v_{n}(a) \right)=1$
are considered together in Subsection~\ref{subsec:u_2n gcd(u_n,v_n)=1}.

\subsection{$a$ is odd and $n \equiv 0 \pmod{3}$}
\label{subsec:u_2n a odd and n=0 mod 3}

The purpose of this subsection is to prove the following result.

\begin{lemma}
\label{lem:u_2n not in pS, a odd, n=0(mod3)}
If $a$ is odd and $n\equiv 0 \pmod{3}$, then $u_{2n}(a) \in p\rS$ is impossible.
\end{lemma}

This lemma follows immediately from the two lemmas below.

\begin{lemma}
\label{lem:2n a odd n=0 mod3 n even}
If $a$ is odd, $n$ is even and $n\equiv 0 \pmod{3}$, then 
$u_{2n}(a) \in p\rS$ is impossible. 
\end{lemma}

\begin{proof}
Write $2n=4m$, where $3 \mid m$. Now $u_{4m}(a) \in p\rS$ implies 
$u_{2m}(a)v_{2m}(a) \in p\rS$, where $\gcd \left( u_{2m}(a), v_{2m}(a) \right)=2$
in view of the assumptions on $a$ and $m$. Therefore, either 
$\left( u_{2m}(a) \in 2\rS \right.$ and $\left. v_{2m}(a) \in 2p\rS \right)$
or $\left( u_{2m}(a) \in 2p\rS \right.$ and $\left.v_{2m}(a) \in 2\rS \right)$.
By Proposition~\ref{prop:mp}, $u_{2m}(a) \in 2\rS$ is impossible. Hence
we are left with 
\begin{equation}
\label{eq:u_{2}m in 2pS and v_{2}m in 2s}
u_{2m}(a) \in 2p\rS \quad \text{and} \quad v_{2m}(a) \in 2\rS.
\end{equation}

From the first relationship in \eqref{eq:u_{2}m in 2pS and v_{2}m in 2s}, we
have $u_{m}(a)v_{m}(a) \in 2p\rS$. Since $a$ is odd and $3 \mid m$, we conclude
that $\gcd \left( u_{m}(a), v_{m}(a) \right)=2$ and $v_{m}(a) \equiv 2 \pmod{4}$. 
Therefore $u_{m}(a) \cdot \frac{1}{2}v_{m}(a) \in p\rS$ and the factors on the 
left-hand side are relatively prime. Consequently, either 
$\left( u_{m}(a) \in \rS \right.$ and $\left. v_{m}(a) \in 2p\rS \right)$ or
$\left( u_{m}(a) \in p\rS \right.$ and $\left. v_{m}(a) \in 2\rS \right)$.

\emph{Consider the first alternative:} $u_{m}(a) \in \rS$ and $v_{m}(a) \in 2p\rS$.
If $m>3$, then, by Proposition~\ref{prop:mp}, we must necessarily have $m=6$ and
$a=3$. But then,
$v_{m}(a)=v_{6}(3)=2\cdot 7\cdot 23$ which is not of the form
$2\cdot\text{prime}\cdot\text{square}$. If $m=3$, then $u_{3}(a) \in \rS$ is
equivalent to $a^{2}-1\in \rS$, which is impossible.

\emph{Consider the second alternative:} $u_{m}(a) \in p\rS$ and $v_{m}(a) \in 2\rS$.
Here, $v_{m}(a)=2x^{2}$ for a positive integer $x$. By the second relation in 
\eqref{eq:u_{2}m in 2pS and v_{2}m in 2s} we also have $v_{2m}(a)=2y^{2}$ for a
positive integer $y$. Using \eqref{eq:vn-closed} and $\alpha\beta=1$, we also
have the relation $v_{m}^{2}(a)-v_{2m}(a)=2$. Hence $2x^{4}-y^{2}=1$ and it is
well-known (see, for example, \cite[page~271]{Mor} and the references there)
that the only positive solutions to this equation are
$(x,y)=(1,1)$ and $(13,239)$. Thus, $v_{m}(a) \in \left\{ 2,\,2\cdot 13^{2} \right\}$.
From Remark~\ref{rem:increasing}, $v_{m}(a) \neq 2$ because $m\geq 3$. Next, consider 
$v_{m}(a) = 2 \cdot 13^{2}=338$. This is impossible for $m\geq 3$.
Indeed, using mathematical software like Maple or PARI/GP, it is easy to check
that $v_{m}(a)=338$ is impossible if $m\in\{3,4\}$ or $(m,a)=(5,3)$. If $m \geq 5$
and $a>3$, then $v_{m}(a) \geq v_{5}(a)=a^{5}-5a^{3}+5a\geq v_{5}(4)=724$, from
Remark~\ref{rem:increasing}.
\end{proof}

\begin{lemma}
\label{lem:2n a odd n=0 mod3 n odd}
If $a$ is odd, $n$ is odd and $n\equiv 0 \pmod{3}$, then $u_{2n}(a) \in p\rS$ is
impossible.
\end{lemma}

\begin{proof}
If $u_{2n}(a) \in p\rS$, then we have $u_{n}(a)v_{n}(a) \in p\rS$ and 
$\gcd \left( u_{n}(a), v_{n}(a) \right)=2$. Therefore, either 
$\left( u_{n}(a) \in 2\rS \right.$ and $\left. v_{n}(a) \in 2p\rS \right)$ or
$\left( u_{n}(a) \in 2p\rS \right.$ and $\left. v_{n}(a) \in 2\rS \right)$.

\emph{Consider the first alternative:} $u_{n}(a) \in 2\rS$ and $v_{n}(a) \in 2p\rS$.
If $n>3$, then $u_{n}(a) \in 2\rS$ is impossible by Proposition~\ref{prop:mp}.
Therefore $n=3$ and 
\begin{equation}
\label{eq:u3=2x^2 and v3=2py^2}
u_{3}(a)=a^{2}-1=2x^{2},\quad \text{and} \quad v_{3}(a)=a^{3}-3a=2py^{2}.
\end{equation}

By the second equation, we have $a\left( a^{2}-3 \right)=2py^{2}$.

If $3 \nmid a$, then $\gcd \left( a,a^{2}-3 \right)=1$. Also, $a$ is odd. Therefore,
either $\left( a=py_{1}^{2} \right.$ and $\left. a^{2}-3=2y_{2}^{2} \right)$ or
$\left( a=y_{1}^{2} \right.$ and $\left. a^{2}-3=2py_{2}^{2} \right)$.

In the first case, combining the first equation 
\eqref{eq:u3=2x^2 and v3=2py^2} with $a^{2}-3=2y_{2}^{2}$, we obtain 
$x^{2}-y_{2}^{2}=1$, which forces
$x=1,y_{2}=0$ and $a^{2}=3$, which is impossible.

In the second case, combining the first equation~\eqref{eq:u3=2x^2 and v3=2py^2}
with $a=y_{1}^{2}$ gives $y_{1}^{4}-1=2x^{2}$, which defines an elliptic curve
of rank $0$. Applying the {\tt IntegralQuarticPoints} function of 
{\sc magma}\cite{magma}\footnote{This routine is an implementation of the method
given in \cite{Tz}.} we find that the only solutions are 
$\left( x,y_{1} \right) = (0,\pm 1)$ leading to $a=1$, which we reject.

If $3 \mid a$, then we put $a=3c$ with $c$ odd, and the 
equations~\eqref{eq:u3=2x^2 and v3=2py^2} become
\begin{equation}
\label{eq:a=3c}
9c^{2}-1=2x^{2},\quad \text{and} \quad c \left( 3c^{2}-1 \right) = 2p(y/3)^{2},
\end{equation}
where, the factors in the left-hand side of the second equation are
relatively prime. Therefore, either $c \in \rS$ or $3c^{2}-1 \in 2\rS$.

In the first case ($c \in \rS$), put $c=z^{2}$ and combine with $9c^{2}-1=2x^{2}$ to obtain
$9z^{4}-1=2x^{2}$. This defines an elliptic curve of rank $1$. Applying the
{\sc magma} function {\tt IntegralQuarticPoints}, we see that the only positive integer
solution is $(c,x)=(1,2)$. But this must be rejected because
it does not satisfy the second of the equations in \eqref{eq:a=3c}.

In the second case (i.e., $3c^{2}-1 \in 2\rS$), we put $3c^{2}-1=2z^{2}$. But we also have 
$9c^{2}-1=2x^{2}$ so, by multiplying these equations we obtain
$\left( 9c^{2}-1 \right) \left( 3c^{2}-1 \right)=(2xz)^{2}$. The equation
$\left( 9x^{2}-1 \right) \left( 3x^{2}-1 \right)=Y^{2}$ defines an elliptic
curve of rank $1$. As before, applying the {\sc magma} function, {\tt IntegralQuarticPoints},
we see that the only integer solution of $\left( 9c^{2}-1 \right) \left( 3c^{2}-1 \right)=(2xz)^{2}$ 
with positive $c$ is $(c,2xz)=(1,4)$. But $c=1$ is rejected as we saw before.

\emph{Consider the second alternative:}
$u_{n}(a) \in 2p\rS$ and $v_{n}(a) \in 2\rS$. We put $n=3m$, where $m$ is odd,
so that $v_{3m}(a) \in 2\rS$. From \eqref{eq:vn-closed} and $\alpha\beta=1$, we
have the identity $v_{3m}(a)=v_{m}^{3}(a)-3v_{m}(a)$. Hence it follows that
$v_{m}(a) \left( v_{m}^{2}(a)-3 \right) \in 2\rS$.

If $3 \nmid v_{m}(a)$, then $\gcd \left( v_{m}(a), v_{m}^{2}(a) - 3 \right)=1$ so, 
either $\left( v_{m}(a) \in 2\rS \right.$ and $\left. v_{m}^{2}(a)-3 \in \rS \right)$ or 
$\left( v_{m}(a) \in \rS \right.$ and $\left. v_{m}^{2}(a)-3 \in 2\rS \right)$.
The second alternative implies $v_{m}(a)=x^{2}$ and $x^{4}-3=2y^{2}$ for some 
$x,y\in\bbZ$, with the last equation being impossible $\bmod \, 9$. 
The first alternative implies $v_{m}(a)=2x^{2}$ and $4x^{4}-3=y^{2}$,
with $x,y$ positive integers. The last equation defines a rank $1$
elliptic curve and the {\sc magma} function {\tt IntegralQuarticPoints} returns $(x,y)=(1,1)$
as the only solution with positive $x,y$.\footnote{Alternatively, since
the coefficient of $x^{4}$ is a square, one can solve the equation by 
elementary means, following the method of \cite{Pou}.\label{foot:Poulakis}} 
Therefore, $v_{m}(a)=2$. But, by Remark~\ref{rem:increasing}, for every $a \geq 3$ and every 
$k>0$ it is true that $v_{k}(a) \geq v_{1}(a)=a>2$. Hence $m=0$ and, consequently, $n=3m=0$, 
a contradiction. 

If $3 \mid v_{m}(a)$, then we put $v_{m}(a)=3x$, so that $x \left( 3x^{2}-1 \right) \in 2S$.
Since $3x^{2}-1 \in \rS$ is impossible $\bmod\,3$, we have $x \in \rS$ and
$3x^{2}-1 \in 2\rS$.
We put $x=x_{1}^{2}$ and $3x^{2}-1=2y^{2}$, so that $3x_{1}^{4}-1=2y^{2}$, which
defines an elliptic curve of rank $1$.
Using the {\sc magma} function {\tt IntegralQuarticPoints} we find that the only solutions with
positive $x_{1}$ are $\left( x_{1},y \right)=(1,1),(3,11)$; consequently $v_{m}(a)/3=x\in\{1,9\}$.
Assume $v_{m}(a)=3$. If $m \geq 3$, then
$3=v_{m}(a) \geq v_{3}(a) \geq v_{3}(3)=18$, contradiction. If $m=1$ (hence $n=3$), 
then $3=v_{m}(a)=v_{1}(a)=a$ and $u_{n}(a)=u_{3}(a)=u_{3}(3)=8 \not\in 2p\rS$.
Assume $v_{m}(a)=27$. If $m=1$ (hence $n=3$), then $27=v_{1}(a)=a$ and 
$u_{n}(a)=u_{3}(27)=728\not\in 2p\rS$. If $m=3$, then $27=v_{3}(a)=a^{3}-3a$ which is 
impossible with $a\in\bbZ$. Finally, if $m \geq 5$, then
$27=v_{m}(a) \geq v_{5}(a) \geq v_{5}(3)=123$, contradiction.
\end{proof} 

\subsection{Both $a$ and $n$ are even}
\label{subsec:u_2n a,n even}

The aim of this subsection is to prove the following lemma, used in the proof
of Proposition~\ref{prop:u_2n in pS}.

\begin{lemma}
\label{lem:u_2n a and n even}
If both $a$ and $n$ are even, then $u_{2n}(a) \in p\rS$ is impossible.
\end{lemma}

\begin{proof}
We put $2n=2^{t}m$, where $t\geq 2$, $m$ is odd and, if $t=2$ then $m \geq 3$.
We have $u_{2n}(a)=u_{n}(a)v_{n}(a) \in p\rS$ and $\gcd \left( u_{n}(a), v_{n}(a) \right)=2$,
by Lemma~\ref{lem:gcd}. Therefore, either
\begin{equation}
\label{eq:u in 2S, v in 2pS}
u_{2^{t-1}m}(a) \in 2\rS \quad \text{and} \quad v_{2^{t-1}m}(a) \in 2p\rS
\end{equation}
or
\begin{equation}
\label{eq:u in 2pS, v in 2S}
u_{2^{t-1}m}(a) \in 2p\rS \quad \text{and} \quad v_{2^{t-1}m}(a) \in 2\rS.
\end{equation}

Since $2^{t-1}m>3$, the first equation~\eqref{eq:u in 2S, v in 2pS} is 
impossible by Proposition~\ref{prop:mp}.

Now we consider \eqref{eq:u in 2pS, v in 2S}. The first equation is
written 
\begin{equation}
\label{eq:uv in 2pS}
u_{2^{t-2}m}(a)v_{2^{t-2}m}(a) \in 2p\rS.
\end{equation}

If $t=2$, then $m\geq 3$ and $u_{m}(a)v_{m}(a) \in 2p\rS$, where the two factors in 
left are relatively prime (cf.~\eqref{eq:gcd u_{n},v_{n}}) and $u_{m}(a)$ is odd.
Thus, either $\left( u_{m}(a) \in \rS \right.$ and $\left. v_{m}(a) \in 2p\rS \right)$ or
$\left( u_{m}(a) \in p\rS \right.$ and $\left. v_{m}(a) \in 2\rS \right)$.

Consider the first alternative. If $m>3$, then, since $m$ is odd, 
$u_{m}(a) \in \rS$ is impossible by Proposition~\ref{prop:mp}. If $m=3$,
then $u_{3}(a) \in \rS$ is equivalent to $a^{2}-1 \in \rS$ which is impossible for 
$a>1$.

Consider the second alternative. Now $v_{m}(a) \in 2\rS$ and, by 
\eqref{eq:u in 2pS, v in 2S}, we already have $v_{2m}(a) \in 2\rS$. 
In the proof of Lemma~\ref{lem:2n a odd n=0 mod3 n even} we have shown
(without using the parity of $a$) that $v_{m}(a) \in 2\rS$ and $v_{2m}(a) \in 2\rS$, with $m \geq 3$,
cannot hold simultaneously. 

If $t>2$ then we go back to \eqref{eq:uv in 2pS}, where now the $\gcd$
of the two factors in the left is 2 and $v_{2^{t-2}m}(a) \equiv 2 \pmod{4}$.
Therefore $u_{2^{t-2}m}(a) \cdot \frac{1}{2}v_{2^{t-2}m}(a) \in p\rS$ and the
factors in the left are relatively prime. Consequently, either
$\left( u_{2^{t-2}m}(a) \in p\rS \right.$ and $\left. v_{2^{t-2}m}(a) \in 2\rS \right)$ or
$\left( u_{2^{t-2}m}(a) \in \rS \right.$ and $\left. v_{2^{t-2}m}(a) \in 2p\rS \right)$.

Consider the first alternative. Then $v_{2^{t-2}m}(a) \in 2\rS$ and, by
\eqref{eq:u in 2pS, v in 2S}, $v_{2^{t-1}m}(a) \in 2\rS$. We put
$2^{t-2}m=k$, so that $k \geq 2$ is even, $v_{k}(a)=2x^{2}$ and $v_{2k}(a)=2y^{2}$.
The case $k\geq 3$ is excluded by appealing again to the proof of 
Lemma~\ref{lem:2n a odd n=0 mod3 n even}. The case $k=2$ implies
$t=3, m=1$ and, by \eqref{eq:u in 2pS, v in 2S}, $u_{4}(a) \in 2p\rS$ and
$v_{4}(a) \in 2\rS$. The last relation is equivalent to $a^{4}-4a^{2}+2=2y^{2}$ 
(note that the left-hand side is $v_{4}(a)$) and, by applying {\tt IntegralQuarticPoints} of
{\sc magma}, we see that the only positive integer solution is 
$(a,y)=(2,1)$, which is rejected since we have assumed $a>2$.

Consider the second alternative. Then $u_{2^{t-2}m}(a) \in \rS$. Since we
have assumed $t>2$, we have either $2^{t-2}m=2$ or $2^{t-2}m>3$.
If $2^{t-2}m=2$, then $t=3,m=1$, so that, by \eqref{eq:u in 2pS, v in 2S},
$v_{4}(a) \in 2\rS$. But a few lines above we saw that $v_{4}(a) \in 2\rS$ is impossible
for $a>2$.
If $2^{t-2}m>3$, then, since $a$ is even, Proposition~\ref{prop:mp} implies that 
$2^{t-2}m=4$ and $a=338$. However, it is straightforward 
to check that $v_{4}(338) \not\in 2p\rS$.
\end{proof}

\subsection{Either $a$ is even and $n$ is odd, or $a$ is odd and $3\nmid n$}
\label{subsec:u_2n gcd(u_n,v_n)=1}

Our purpose in this subsection is to prove 
Lemma~\ref{lem:u_2n in pS => u_m in pS}. As we wrote above, this lemma
is necessary for the proof of Proposition~\ref{prop:u_2n in pS}.	

\begin{lemma}
\label{lem:u_2n in pS => u_{n} in pS}
If $u_{2n}(a) \in p\rS$, where $n \geq 3$ with $a$ and $n$ as in the title of
\emph{Subsection~\ref{subsec:u_2n gcd(u_n,v_n)=1}}, then $u_{n}(a) \in p\rS$
and $v_{n}(a) \in \rS$.
\end{lemma}

\begin{proof}
We have $u_{n}(a)v_{n}(a)=u_{2n}(a) \in p\rS$ and $\gcd \left( u_{n}(a), v_{n}(a) \right)=1$.
Therefore, either $\left( u_{n}(a) \in \rS \right.$ and $\left. v_{n}(a) \in p\rS \right)$
or $\left( u_{n}(a) \in p\rS \right.$ and $\left. v_{n}(a) \in \rS \right)$, and
we have to show that the first alternative is impossible. 
If $n>3$, then, by Proposition~\ref{prop:mp} and $u_{n}(a) \in \rS$, we conclude
that $(n,a) \in \{ (4,338),\,(6,3) \}$. But, according to our assumption, if
$a$ is even, $n$ must be odd, and if $a$ is odd, $n$ must be indivisible by $3$.
Therefore both possibilities for $(a,n)$ are excluded and, consequently, we must
reject $n>3$.
If $n=3$, then $u_{3}(a) \in \rS$ means $a^{2}-1 \in \rS$, which is impossible
for $a>1$.
\end{proof}

\begin{lemma}
\label{lem:u_2n in pS => u_m in pS}
If $u_{2n}(a) \in p\rS$, where $n \geq 3$ with $a$ and $n$ as in the title of
\emph{Subsection~\ref{subsec:u_2n gcd(u_n,v_n)=1}}.
Then $n$ has odd divisors greater than $1$, and if $m>1$ is the largest odd divisor
of $n$, then $u_{m}(a) \in p\rS$ and $v_{m}(a) \in \rS$.
\end{lemma}

\begin{proof}
We first show that $n$ cannot be equal to a power of $2$. Assume the contrary
and put $2n=2^{t}$, where $t \geq 3$. Before proceeding, we observe that
$v_{4}(a) \in \rS$ is impossible, because
it is equivalent to $a^{4}-4a^{2}+2=x^{2}$, which is impossible $\bmod\,3$.
We have $u_{2\cdot 2^{t-1}}(a)=u_{2^{t}}(a) \in p\rS$ and $2^{t-1} \geq 4$.
Therefore, by Lemma~\ref{lem:u_2n in pS => u_{n} in pS}, 
$u_{2^{t-1}}(a) \in p\rS$ and $v_{2^{t-1}}(a) \in \rS$.
If $t=3$, then $v_{4}(a) \in \rS$, which is impossible. If $t \geq 4$, then,
$u_{2\cdot 2^{t-2}}(a)=u_{2^{t-1}}(a) \in p\rS$ and $2^{t-1} \geq 4$. It
follows by Lemma~\ref{lem:u_2n in pS => u_{n} in pS} that 
$u_{2^{t-2}}(a) \in p\rS$ and $v_{2^{t-2}}(a) \in \rS$. If $t=4$, then
$v_{4}(a) \in \rS$, which is impossible. If $t>4$, we repeat the above argument
so that, after a finite number of steps we obtain the impossible relation
$v_{4}(a) \in \rS$. 
This proves the first claim of the lemma. 

To prove the second claim of the lemma, we put $2n=2^{t}m$, where $t \geq 1$ 
and $m>1$ is odd.
If $a$ is even, then, by hypothesis, $n$ is odd. Hence $t=1, n=m$, and 
Lemma~\ref{lem:u_2n in pS => u_{n} in pS} imply that $u_{m}(a) \in p\rS$
and $v_{m}(a) \in \rS$.
Next, assume that $a$ is odd. By hypothesis, $3\nmid n$. Thus $3\nmid m$.
We have $u_{2\cdot 2^{t-1}m}(a) \in p\rS$. Hence, by 
Lemma~\ref{lem:u_2n in pS => u_{n} in pS}, $u_{2^{t-1}m}(a) \in p\rS$ and 
$v_{2^{t-1}m}(a) \in \rS$. If $t=1$ we are done. If $t>1$, then 
$u_{2\cdot 2^{t-2}m}(a) \in p\rS$ and Lemma~\ref{lem:u_2n in pS => u_{n} in pS}
implies that $u_{2^{t-2}m}(a) \in p\rS$ and $v_{2^{t-2}m}(a) \in \rS$. If
$t=2$, then we are done. If $t>2$, we repeat the argument until we arrive at
$u_{m}(a) \in p\rS$ and $v_{m}(a) \in \rS$.
\end{proof}

\subsection{Proof of Proposition~\ref{prop:u_2n in pS}}
\label{subsec:prop u_2n in pS proof}

From Lemma~\ref{lem:u_2n a and n even}, we see that either $a$ or $n$ is odd.

If $a$ is odd, then, from Lemma~\ref{lem:u_2n not in pS, a odd, n=0(mod3)},
$n \not\equiv 0 \pmod{3}$. From Lemma~\ref{lem:u_2n in pS => u_m in pS},
$n$ must have odd divisors greater than $1$ and if we let $m$ be the largest such
odd divisor of $n$, then $u_{m}(a) \in p\rS$ and $v_{m}(a) \in \rS$. 
By Lemma~\ref{lem:RibMcDan}(c) (in the notation of which $b=-1$ in our case), $v_{m}(a) \in \rS$ 
implies that $m=3$ or $m=5$.
Since $n \not\equiv 0 \pmod{3}$, the only possibility is $m=5$. 
But for $m=5$, Lemma~\ref{lem:RibMcDan}(c) says that $-b\equiv 5 \pmod{8}$,
a contradiction because, in our case $-b=1$. 
Thus, if $a$ is odd, then $u_{2n}(a)\in p\rS$ is not possible.

If $n$ is odd, then, by what we proved above, $a$ is even and we can again use 
Lemma~\ref{lem:u_2n in pS => u_m in pS} to conclude that there exists
an odd $m>1$ such that $v_m(a)\in\rS$. 
As in the previous paragraph, we see that $m=5$ is not possible. If $m=3$, then
$-b \equiv 3 \pmod{4}$ by Lemma~\ref{lem:RibMcDan}(c). But in our case,
$-b=1$ and we get a contradiction. 
Hence if $n$ is odd, then $u_{2n}(a)\in p\rS$ is not possible.

\section{$u_{2n+1}(a)\in p\rS$ for $a \geq 3$, $n \geq 2$ and $p \geq 5$}
\label{sec:u_2n+1 in pS}

\emph{
Throughout this section, we always assume that $a\geq 3$, $n\geq 2$ and the
prime $p$ satisfies $p \geq 5$. No further mention of these assumptions will be
made in this section.}

\begin{lemma}
\label{lem:t_n or w_n must be square}
Assume that either $a$ is even or $(a$ is odd and $n \not\equiv 1\pmod{3})$.
If $u_{2n+1}(a)\in p\rS$, then either $w_{n}(a)$ or $t_{n}(a)$ is a square $>1$.
\end{lemma}

\begin{proof}
From relation~\eqref{eq:u=t*w}, we have $u_{2n+1}(a) = t_{n}(a)w_{n}(a)\in p\rS$.
Under our assumptions on $a$, Lemma~\ref{lem:gcd}(b) implies that 
$\gcd \left( t_{n}(a), w_{n}(a) \right)=1$,
Hence either $w_{n}(a)$ or $t_{n}(a)$ is a square which, moreover,
is $>1$ because of the assumptions that $a\geq 3$ and $n\geq 2$.
\end{proof}

In the remainder of this section, our treatment of 
\begin{equation}
\label{eq:u_2n+1 in pS}
u_{2n+1}(a) \in p\rS
\end{equation}
will be as follows.

If $a$ is odd and $n \equiv 1 \pmod{3}$, then we will show that \eqref{eq:u_2n+1 in pS}
is impossible in Subsection~\ref{subsec:a odd and n=1 mod 3}
(see Lemma~\ref{lem:2n+1 n=1(mod 3) a odd}).

If $a$ is odd and $n\not\equiv 1 \pmod{3}$, then we will show that \eqref{eq:u_2n+1 in pS}
is impossible in Subsubsection~\ref{subsubsec:a_odd_n_not_1_mod3}
(see Proposition~\ref{prop:2n+1 n<>1(mod 3) a odd}).

As a consequence of the above two statements, we have the following conclusion
for odd $a$.

\begin{proposition}
\label{prop:u_2n+1 not square with odd a}
$u_{2n+1}(a)\in p\rS$ is impossible with odd $a$.	
\end{proposition}

If $a$ is even, then, by Lemma~\ref{lem:t_n or w_n must be square}, either
$w_{n}(a) \in \rS$ or $t_{n}(a) \in \rS$. By Lemma~\ref{lem:w_n in S},
$w_{n}(a) \in \rS$ is impossible.
Thus, we are left with $t_{n}(a)\in\rS$ when $a$ is even. The study of this relation is done 
in Subsubsection~\ref{subsubsec:a_even}, where we prove 
Lemma~\ref{lem:t_n square,general a}. As a consequence of that proposition
and the fact that
$w_{n}(a)\in\rS$ is impossible, we conclude that, if $a$ is even, then 
\eqref{eq:u_2n+1 in pS} is possible only if $a \equiv 2 \pmod{4}$, and even more: For 
$a \equiv 2 \pmod{4}$ there exists at most one $n$ such that \eqref{eq:u_2n+1 in pS} is true and,
if such an $n$ actually exists, then it is odd, $\geq 9$ and $2n+1$ is a prime.

%%%%%%%%%%%%%%%%%%%%%%%%%%%%%%%%%%%%%

\subsection{$a$ is odd and $n\equiv 1 \pmod{3}$}
\label{subsec:a odd and n=1 mod 3}

The aim of this subsection is to prove the following lemma.

\begin{lemma}
\label{lem:2n+1 n=1(mod 3) a odd}
If $a $ is odd and $n \equiv 1 \pmod{3}$, then $u_{2n+1}(a) \in p\rS$ is impossible.
\end{lemma}

\begin{proof} 
Assume that our claim is not true and let $2n+1$ be the least odd multiple of $3$, 
strictly greater than $3$, such that $u_{2n+1}(a) \in p\rS$.

Instead of treating this equation using the relation $u_{2n+1}(a)=t_{n}(a)w_{n}(a)$, we
write $2n+1=3m$, so that 
$u_{2n+1}(a)=u_{3m}(a)=u_{m}(a) \left( v_{m}(a)+1 \right) \left( v_{m}(a)-1 \right)$
(note that, by our assumption, $m$ is odd $>1$). Therefore, $u_{2n+1}(a) \in p\rS$
becomes
\begin{equation}
\label{eq: u_m(v_m+1)(v_m-1)=px^{2}}
u_{m}(a) \left( v_{m}(a)+1 \right) \left( v_{m}(a)-1 \right) \in p\rS.
\end{equation}

We have
\begin{align}
\label{eq: gcd of u_m,v_m+1,v_m-1}
& \left( \gcd \left( u_{m}(a), v_{m}(a)+1 \right), \gcd \left( u_{m}(a), v_{m}(a)-1 \right),
         \gcd \left( v_{m}(a)+1, v_{m}(a)-1 \right) \right) \\
& =
\begin{cases}
   (1,1,2) & \text{if $m\not\equiv 0\pmod{3}$} \\
   (1,1,1) & \text{if $m\equiv 0\pmod{3}$ and $a\equiv 0\pmod{3}$} \\
   (1,3,1) & \text{if $m\equiv 0\pmod{3}$ and $a\equiv 1\pmod{3}$} \\
   (3,1,1) & \text{if $m\equiv 0\pmod{3}$ and $a\equiv 2\pmod{3}$}. 
\end{cases}
\nonumber
\end{align}

The proof is easy. The following two remarks are helpful.\\
(i) From \eqref{eq:un2-vn2-relationship}, we have $v_{m}^{2}(a)-\Delta u_{m}^{2}(a)=4$.
So $\left( v_{m}(a)+1 \right) \left( v_{m}(a)-1 \right) - \Delta u_{m}^{2}(a)=3$.
Therefore $\gcd \left( u_{m}(a), v_{m}(a) \pm 1 \right)=1$ or $3$.\\
(ii) If $a$ is odd and $3 \mid m$, then $v_{m}(a)$ is even. Therefore
$\gcd \left( v_{m}(a)+1,v_{m}(a)-1 \right)=1$.

\emph{If the first case of \eqref{eq: gcd of u_m,v_m+1,v_m-1} holds}, 
then, \eqref{eq: u_m(v_m+1)(v_m-1)=px^{2}}
is equivalent to 
$u_{m}(a) \cdot \frac{1}{2} \left( v_{m}(a)+1 \right) \cdot \frac{1}{2} \left( v_{m}(a)-1 \right) 
\in p\rS$,
where the three factors in the left-hand side are relatively prime in pairs. Therefore, 
exactly one out of the three belongs to $p\rS$ and each of the remaining two belongs 
to $\rS$. If any of the last two factors belongs to $p\rS$, then $u_{m}(a) \in \rS$, which 
implies, by Proposition~\ref{prop:mp}, that $(a,m) \in \{ (3,6),(338,4) \}$; both 
alternatives are rejected since both $a$ and $m$ are odd.
Therefore, both $\frac{1}{2} \left( v_{m}(a)-1 \right)$ and $\frac{1}{2} \left( v_{m}(a)+1 \right)$
belong to $\rS$. 
But then, $v_{m}(a)+1=2x_{1}^{2}$, $v_{m}(a)-1=2x_{2}^{2}$. 
Hence $x_{1}^{2}-x_{2}^{2}=1$.
This implies $x_{2}=0$ which is impossible, since $x_{2}^{2}=v_{m}(a)-1>0$.

\emph{If the second case of \eqref{eq: gcd of u_m,v_m+1,v_m-1} holds}, then, among 
the three factors in the left-hand side of \eqref{eq: u_m(v_m+1)(v_m-1)=px^{2}} exactly 
one belongs to $p\rS$ and each of the remaining two factors belongs to $\rS$. If any of 
the last two factors belongs to $p\rS$, then $u_{m}(a) \in \rS$, which is impossible as 
before.
Therefore, both $v_{m}(a) \pm 1$ belong to $\rS$. However, the simultaneous equations
$v_{m}(a)+1=x_{1}^{2}$, $v_{m}(a)-1=x_{2}^{2}$ imply $x_{1}^{2}-x_{2}^{2}=2$ which is 
impossible $\bmod\,4$.

\emph{If the third case of \eqref{eq: gcd of u_m,v_m+1,v_m-1} holds}, then, 
\eqref{eq: u_m(v_m+1)(v_m-1)=px^{2}} is equivalent to 
$\frac{1}{3}u_{m}(a) \cdot \left( v_{m}(a)+1 \right) \cdot \frac{1}{3} \left( v_{m}(a)-1 \right)
 \in p\rS$,
where the three factors in the left-hand side are relatively prime in pairs.
Therefore, exactly one out of the three belongs to $p\rS$ and each of the remaining 
two belongs to $\rS$.
If any of the last two factors belongs to $p\rS$, then
$\frac{1}{3}u_{m}(a) \in \rS$. Hence $u_{m}(a) \in 3\rS$. By Proposition~\ref{prop:mp}, $m=3$
and $2n+1=9$. Then, in view of our hypothesis, $u_9(a) \in p\rS$ which we write 
$\left( a^{4}-a^{3}-3a^{2}+2a+1 \right) \left( a^{4}+a^{3}-3a^{2}-2a+1 \right) \in p\rS$.
It is easy to see that $2$ is the $\gcd$ of the two factors. Hence 
\[
\frac{a^{4}-a^{3}-3a^{2}+2a+1}{2}\cdot
\frac{a^{4}+a^{3}-3a^{2}-2a+1}{2} \in p\rS
\]
and the two factors are relatively prime. It follows that exactly one of them belongs
to $\rS$ i.e. either $a^{4}-a^{3}-3a^{2}+2a+1\in 2\rS$ or 
$a^{4}+a^{3}-3a^{2}-2a+1\in 2\rS$.
Each of these relations defines an elliptic curve of rank $0$ and we find out that 
$(0,0)$ is the only integral point, which is rejected. 
Therefore, $\frac{1}{3}u_{m}(a) \in p\rS$ and 
$\frac{1}{3} \left( v_{m}(a)-1 \right), v_{m}(a)+1 \in \rS$. 
But then, $v_{m}(a)+1=x_{1}^{2}$ and $v_{m}(a)-1=3x_{2}^{2}$. 
Hence $x_{1}^{2}-3x_{2}^{2}=2$, which is impossible $\bmod\,3$.

\emph{If the fourth case of \eqref{eq: gcd of u_m,v_m+1,v_m-1} holds}, then, like in 
the third case, just above, we conclude that we simultaneously have $u_{m}(a) \in 3p\rS$, 
$v_{m}(a)+1 \in 3\rS$ and $v_{m}(a)-1 \in \rS$.
Now, the simultaneous equations $v_{m}(a)+1=3x_{1}^{2}$, $v_{m}(a)-1=x_{2}^{2}$ cannot be 
proved impossible by a straightforward argument, like in the previous three cases. 
Therefore we will argue differently. Before proceeding, we remark that
$v_{m}(a)-1=x_{2}^{2}$ cannot hold with $m=3$, because it is equivalent to 
$a^{3}-3a-1=x_{2}^{2}$, which has no integer solutions with $a>2$ (its only
integer solutions are given by $a=-1,2$). This can be seen by applying the
{\tt IntegralPoints} function of {\sc magma}. Therefore, $m>3$. Further, 
since we deal with the fourth case of \eqref{eq: gcd of u_m,v_m+1,v_m-1}, we have 
$a\equiv 2 \pmod{3}$ and $m=3k$, where both $a \geq 3$ and $k>1$ are odd. We have
$u_{m}(a)=u_{3k}(a) = u_{k}(a) \left( v_{k}(a)+1 \right) \left( v_{k}(a)-1 \right)$.
Hence,
\begin{equation}
\label{eq:u_k(v_k+1)(v_k-1) in pS}
u_{k}(a) \left( v_{k}(a)+1 \right) \left( v_{k}(a)-1 \right) \in 3p\rS.
\end{equation}

The left-hand side is similar to the left-hand side of 
\eqref{eq: u_m(v_m+1)(v_m-1)=px^{2}} with $k$ in place 
of $m$. Therefore, \eqref{eq: gcd of u_m,v_m+1,v_m-1} is still valid, with $k$ in place 
of $m$, but now, since $a\equiv 2 \pmod{3}$, we have to consider only the first and the 
fourth case. I.e. either 
\begin{align*}
\text{(i)}\; k \not\equiv 0 & \pmod{3} 
\\ 
     & \text{ and }
        \left( 
        \gcd ( u_{k}(a), v_{k}(a)+1 ), 
        \gcd ( u_{k}(a), v_{k}(a)-1 ), 
        \gcd ( v_{k}(a)+1, v_{k}(a)-1 ) 
        \right)=(1,1,2) \\
        \text{or}  \\
\text{(ii)}\; k \equiv 0  & \pmod{3} \\
    &  \text{ and }
      \left( 
      \gcd ( u_{k}(a), v_{k}(a)+1 ),
      \gcd ( u_{k}(a), v_{k}(a)-1 ), 
      \gcd ( v_{k}(a)+1, v_{k}(a)-1 ) 
      \right)=(3,1,1).
\end{align*}

If (i) holds, then $u_{k}(a) \not\equiv 0\pmod{3}$. 
We can write \eqref{eq:u_k(v_k+1)(v_k-1) in pS} as
$u_{k}(a) \cdot \frac{1}{2} \left( v_{k}(a)+1 \right) \cdot \frac{1}{2}
 \left( v_{k}(a)-1 \right) \in 3p\rS$
where the three factors in the left are pairwise co-prime. Therefore, $u_{k}(a) \in c\rS$,
where $c \in \{ 1,p \}$. But, since $k$ is odd, we must exclude $c=1$, in view of
Proposition~\ref{prop:mp}. Therefore $u_{k}(a) \in p\rS$. Also, since 
$a \equiv 2\pmod{3}$, we have $v_{k}(a)+1 \equiv 0 \pmod{3}$. Therefore 
$\frac{1}{2} \left( v_{k}(a)+1 \right)=3x_{1}^{2}$ and 
$\frac{1}{2} \left( v_{k}(a)-1 \right)=x_{2}^{2}$.
From the last two relations we obtain $1=3x_{1}^{2}-x_{2}^{2}$, which is impossible 
$\bmod\,3$.

If (ii) holds, then $u_{k}(a) \equiv 0 \pmod{3}$ and $v_{k}(a)+1 \equiv 3 \pmod{9}$ and 
$v_{k}(a)-1 \not\equiv 0\pmod{3}$.
We write \eqref{eq:u_k(v_k+1)(v_k-1) in pS}, 
$u_{k}(a) \cdot \frac{1}{3} \left( v_{k}(a)+1 \right) \cdot \left( v_{k}(a)-1 \right) \in p\rS$,
where the three factors in the left are pairwise co-prime. Consequently, 
$u_{k}(a) \in c\rS$, where $c \in \{1,p\}$. 
But, as before, $c=1$ is excluded by Proposition~\ref{prop:mp}.
Hence $u_{k}(a) \in p\rS$ and $k$ is an odd multiple of $3$.
If $k>3$, then we are in contradiction with the minimality of $2n+1$ (cf.~beginning of 
the proof) since $2n+1=3m=9k>k$. If $k=3$, then $u_{3}(a) \in p\rS$ and
$\frac{1}{3} \left( v_{3}(a)+1 \right) \in \rS$ and $v_{3}(a)-1 \in \rS$. The last relation
is equivalent to $a^{3}-3a-1=y^{2}$. As we have seen earlier in the proof, there
are no integer solutions $(a,y)$ with $a>2$.
\end{proof}

\subsection{Either $a$ is even or $n\not\equiv 1\pmod{3}$}
\label{subsec:a even or n=0,2 mod 3}

If $a$ is odd (hence $n \not\equiv 1 \pmod{3}$ here), the relation $u_{2n+1}(a)\in p\rS$
is impossible by Proposition~\ref{prop:2n+1 n<>1(mod 3) a odd} in 
Subsubsection~\ref{subsubsec:a_odd_n_not_1_mod3} below.
If $a$ is even, then, in view of Lemma~\ref{lem:t_n or w_n must be square},
the relation $u_{2n+1}(a)\in p\rS$ implies that either $t_{n}(a)$
or $w_{n}(a)$ is a square.
In the particular case of this subsection (namely, either $a$ is even or $n\not\equiv 1\pmod{3}$),
it is impossible for $w_{n}(a)$ to be a square in view of Lemma~\ref{lem:w_n in S},
which we prove below. Therefore, $u_{2n+1}(a)\in p\rS$ is reduced to $t_{n}(a)\in\rS$,
which (for even $a$) is treated in Subsubsection~\ref{subsubsec:a_even}. There
we prove Lemmas~\ref{lem:t_5=square}, \ref{lem:TogbeVoutierWalsh} and 
\ref{lem:t_n square,general a}.
This last lemma is essentially Theorem~\ref{thm:at most one solution}(iii).

%An important identity that will play a crucial role in this subsection is
%\begin{equation}
%\label{eq:main identity}
%(a+2)t_{n}(a)^{2} - (a-2)w_{n}(a)^{2} = 4.
%\end{equation}
%
%This is equation~\eqref{eq:tn2-wn2-relationship} in Lemma~\ref{lem:seq-relationships}
%applied with $b_{1}=1$.

In the proof of Lemma~\ref{lem:w_n in S} we will use the following three
lemmas.

\begin{lemma}
\label{lem:x^2-(c^2-4)y^2=-4}
For odd $c>3$, the equation 
\[
X^{2} - \left( c^{2}-4 \right)Y^{2} = -4
\]	
has no integer solutions in relatively prime (hence odd) $X,Y$.
\end{lemma}

\begin{proof}
If there is an integer solution $(X,Y)$, then $XY\neq 0$. Assume that $(X,Y)$ is
a positive solution. Since $4<\sqrt{c^{2}-4}$, it follows that $X/Y=p_{k}/q_{k}$
for some $k\geq 0$, where $p_{i}/q_{i}$ ($i=0,1,2,\ldots$) denote the convergents 
of the continued fraction expansion of $\sqrt{c^{2}-4}$ 
(see, for example, Theorem~3.3 of \cite{SolvePell}).
By hypothesis, $X$ and $Y$ are relatively prime, hence $X=p_{k}$ and $Y=q_{k}$.
Therefore, $p_{k}^{2} - \left( c^{2}-4 \right)q_{k}^{2}=-4$. We prove that this
is impossible as follows.
	
Applying Theorem~3.2.1(f) of \cite{MollinQuadratics} with $a=c$ and $r=-4$, we find that for odd 
$c>3$, the continued fraction expansion of $\sqrt{c^{2}-4}$ is
\[
\sqrt{c^{2}-4}=\left\langle c-1;\overline{1,\, (c-3)/2,\, 2,\, (c-3)/2,\, 1,\, 2c-2} \right\rangle.
\]

We have $p_{0}=c-1$ and $q_{0}=1$, so 
$p_{0}^{2} - \left( c^{2}-4 \right)q_{0}^{2}=-2c+5\neq -4$. The sequence
$\left( p_{k}^{2} - \left( c^{2}-4 \right)q_{k}^{2} \right)_{k\geq 1}$ is periodic
with period: $4,-c+2,4,-2c+5,1,-2c+5$.
See again Theorem~3.2.1(f) of \cite{MollinQuadratics}.
Therefore, $p_{k}^{2} - \left( c^{2}-4 \right) q_{k}^{2}\neq -4$.
\end{proof}

\begin{lemma}
\label{lem:x^2-(c^2-4)y^4=+4}
For odd $c \geq 3$, except for $(X,Y)=(c,1)$, there is at most one further positive solution
of the equation
\[
X^{2} - \left( c^{2}-4 \right)Y^{4}=4,
\]
with $XY$ odd.

This additional odd positive solution exists exactly when $c$ is a perfect square. In this case,
the second positive odd solution is $(X,Y) = \left( c^{2}-2,\sqrt{c} \right)$.
\end{lemma}

\begin{proof}
For $c=3$, we can use {\sc magma} to show that the only solutions in positive integers
are $(X,Y)=(3,1)$ and $(322,12)$. So we will assume that $c \geq 5$ in what follows.

By Lemma~\ref{lem:x^2-(c^2-4)y^2=-4}, we can apply Lemma~\ref{lem:Cohn67} with
$d=c^{2}-4$ to our equation. In our case, $\left( a_{1}, b_{1} \right)=(c,1)$.
Therefore, $Y^{2} \in \left\{ 1,c,c^{2}-1 \right\}$. By assumption, $Y$ is odd.
Therefore $Y^{2} \neq c^{2}-1$.
Consequently, one positive solution is obtained from $Y^{2}=1$, in which case $X=c$.
A second solution can come from $Y^{2}=c$ which happens only if $c$ is a perfect
square. In this last case, $X=c^{2}-2$.
\end{proof}

\begin{lemma}
\label{lem:x^2-(c^2-1)y^4=1}
If $c \geq 2$, then $(X,Y)=(c,1)$ is the only solution in positive integers
of the equation 
\begin{equation}
\label{eq:x^2-(c^2-1)y^4=1}
X^{2} - \left( c^{2}-1 \right)Y^{4}=1,
\end{equation}
with $Y$ odd.
\end{lemma}

\begin{proof}
Equation~\eqref{eq:x^2-(c^2-1)y^4=1} is a special case of the equation 
$X^{2}-DY^{4}=1$, where $D$ is a non-square integer. So we can apply
Lemma~\ref{lem:TVW2005} with $D=c^{2}-1$.
We have $\varepsilon_{D}=c+\sqrt{c^{2}-1}$, so $T_{1}=c$ and $U_{1}=1$.
We find that $U_{2}=2c$ and $U_{4}=8c^{3}-4c$. Since we are only considering
solutions with $Y$ odd, we must have $Y^{2}=U_{1}=1$ and consequently $X=c$.
\end{proof}

In the proof of the next proposition we apply results of \cite{Ljun54} concerning
the equation $Ax^{2}-By^{4}=C$, where $C\in\{1,2,4\}$, $A$ and $B$ are positive integers
with $AB$ odd when $C$ is even, $AB$ is not a square and, if $C=1$ or $4$, $A$
is not a square.
For the reader's convenience, the translated results of \cite{Ljun54} are included in 
Appendix~\ref{app:ljunggren}. More precisely, we will apply Theorems~I, II and III of
Appendix~\ref{app:ljunggren}.

\begin{lemma}
\label{lem:w_n in S}
If either $a$ is even or $n\not\equiv 1\pmod{3}$, then $w_{n}(a)\in\rS$ is impossible. 
\end{lemma}

\begin{proof}
By \eqref{eq:gcd(t,w)} we have $\gcd \left( t_{n}(a), w_{n}(a) \right)=1$. Combining
this with \eqref{eq:tn2-wn2-relationship} for $b_{1}=1$, it follows by a simple
parity argument that $t_{n}(a)$
and $w_{n}(a)$ are odd. If $w_{n}(a)\in\rS$ with $n>1$, then, putting $w_{n}(a)=W^{2}$ and 
$t_{n}(a)=T$, with $W>1$, $T>1$, we have, by \eqref{eq:tn2-wn2-relationship} with $b_{1}=1$,
\begin{equation}
\label{eq:(a+2)T^2-(a-2)W^4=4}
(a+2)T^{2}-(a-2)W^{4}=4, \quad \text{$T,W$ odd $>1$ and relatively prime.}
\end{equation}

\emph{First case: $a$ is odd.}
We note that $(a+2)(a-2)$ is not a square for, otherwise, the relatively prime
numbers $a+2$ and $a-2$ must both be squares, which is impossible for $a>2$.
If $a+2$ is not a square, then, for the equation 
$(a+2)x^{2}-(a-2)y^{4}=4$ we may apply Theorem~\ref{thm:III} in 
Appendix~\ref{app:ljunggren} (\cite[Satz III]{Ljun54}) which concerns the equation 
$Ax^{2}-By^{4}=4$.
The conditions of that theorem are fulfilled for $A=a+2, B=a-2$,
and we conclude that there is at most one positive solution $(x,y)$. Since $x=1$,
$y=1$ is actually such a solution, it is the only such solution. 
Hence in \eqref{eq:(a+2)T^2-(a-2)W^4=4} we necessarily have $T=1$ and $W=1$ which
is rejected.

Next, assume that $a+2=c^{2}$ with $c$ odd. 
Then \eqref{eq:(a+2)T^2-(a-2)W^4=4} becomes
\begin{equation}
\label{eq:when a+2=square no 1}
(cT)^{2} - \left( c^{2}-4 \right)W^{4}=4.
\end{equation}

This is a special case of $y^{2}=dx^{4}+4$ with $d=c^{2}-4$,
$y=cT$ and $x=W^{2}$. 
If $c=3$, then we have $d=5$ and the only positive solutions are $(x,y)=(1,3)$
and $(12,322)$ according to Lemma~\ref{lem:Cohn66}.
This implies that, in \eqref{eq:when a+2=square no 1}, $W\in\{1,12\}$ and both are rejected 
because, by \eqref{eq:(a+2)T^2-(a-2)W^4=4}, $W$ is odd $>1$.

If $c>3$, then, by Lemma~\ref{lem:x^2-(c^2-4)y^4=+4}, 
equation~\eqref{eq:when a+2=square no 1} only has positive integer solutions
with $W=1$, which is rejected, or $W^{2}=c$.
In the second case, it follows by \eqref{eq:when a+2=square no 1} that 
$c^{2} \mid 4$. But this is not possible as $c \ge 3$ is odd.

\emph{Second case: $a\equiv 2\pmod{4}$.} 
Now we write equation~\eqref{eq:(a+2)T^2-(a-2)W^4=4} as
\[
\frac{a+2}{4}\cdot T^{2} - \frac{a-2}{4}\cdot W^{4} = 1.
\] 

As in the previous case, $\frac{a+2}{4}\cdot\frac{a-2}{4}$ is not a square. If $(a+2)/4$ 
is not a square, then, applying 
Theorem~\ref{thm:I} of Appendix~\ref{app:ljunggren} (\cite[Satz~I]{Ljun54}) 
or, equivalently, Theorem~3.6 of \cite{Walsh98}, we infer that 
the equation $\frac{a+2}{4}x^{2} -\frac{a-2}{4} y^{4}=1$ has at most one positive 
solution. As $(x,y)=(1,1)$ is such a solution, we conclude that this is the
only positive solution of the equation. This implies that $(T,W)=(1,1)$ which
we reject.
	
If $a+2$ is a square, then $a=4c^{2}-2$ with $c>1$ and \eqref{eq:(a+2)T^2-(a-2)W^4=4} 
becomes
\[
(cT)^{2} - \left( c^{2}-1 \right)W^{4}=1.
\]

Since $W$ is odd, Lemma~\ref{lem:x^2-(c^2-1)y^4=1} implies $W=1$, which we reject.
	
\emph{Third case: $a\equiv 0\pmod{4}$.} 
Now \eqref{eq:(a+2)T^2-(a-2)W^4=4} is written
\[
\frac{a+2}{2}\cdot T^{2} - \frac{a-2}{2}\cdot W^{4} = 2,
\]
where both $(a+2)/2$ and $(a-2)/2$ are odd and their product is not a square. 
We apply Theorem~\ref{thm:II} of Appendix~\ref{app:ljunggren} 
(\cite[Satz~II]{Ljun54}) concerning $Ax^{2}-By^{4}=C$, with $A=(a+2)/2, B=(a-2)/2$
and $C=2$. The least positive solution of $Az_{1}^{2}-Bz_{2}^{2}=C$ is
$\left( a_{0}, b_{0} \right)=(1,1)$ and the 
requirement for the application of Theorem~\ref{thm:II} of Appendix~\ref{app:ljunggren}
is that $\dfrac{4}{C}Bb_{0}^{2}+3$ is not a square. In our case, this quantity
is equal to $a+1$. Therefore, if $a+1$ is not a square then, by Theorem~\ref{thm:II}
of Appendix~\ref{app:ljunggren}, we infer 
that our equation has at most one positive solution. But, with $A,B,C$ as above,
$(x,y)=(1,1)$ is a solution of $Ax^{2}-By^{4}=C$. Therefore this is the only
positive solution, forcing $(T,W)=(1,1)$ which, clearly, we reject.

If $a+1=c^{2}$, then $c\geq 3$ is odd, and \eqref{eq:(a+2)T^2-(a-2)W^4=4} becomes
\begin{equation}
	\label{eq:when a+2=square no 3}
\frac{c^{2}+1}{2}\,T^{2}-\frac{c^{2}-3}{2}\,W^{4}=2,
\end{equation}
where the coefficients of $T^{2}$ and $W^{4}$ are odd and their product is not a square. 
Now we can apply Theorem~\ref{thm:I} of Appendix~\ref{app:ljunggren}
(Satz~I of \cite{Ljun54}) to conclude that
equation~\eqref{eq:when a+2=square no 3} has at most two positive solutions $(T,W)$.
But we see that $(T,W)=(1,1), \left( c^{2}-2,c \right)$ are actual solutions,
and the first one is not acceptable.

Therefore, $T=c^{2}-2$ and $W=c$, which means (by the definition of $T$ and $W$) that, for some
$n>1$, we have $t_{n} \left( c^{2}-1 \right)=c^{2}-2$ and 
$w_{n} \left( c^{2}-1 \right)=W^2=c^2$.
This is impossible because 
$w_{n} \left( c^{2}-1 \right) \geq w_{2} \left( c^{2}-1 \right)=c^{4}-c^{2}-1>c^2$
(see Remark~\ref{rem:increasing}).
\end{proof}

\subsubsection{$a$ is odd and $n\not\equiv 1\pmod{3}$}
\label{subsubsec:a_odd_n_not_1_mod3}

\begin{lemma}
	\label{lem:t_n in S odd a}
If $a$ is odd and $n>1$, then $t_{n}(a)\in\rS$ is impossible.	
\end{lemma}

\begin{proof}
If $t_{n}(a)\in\rS$ with $n>1$, then, putting $t_{n}(a)=T^{2}$ and $w_{n}(a)=W$ 
with $T>1$ and $W>1$, we have \eqref{eq:tn2-wn2-relationship} with $b_{1}=1$,
\begin{equation}
\label{eq:(a+2)T^4-(a-2)W^2=4}
(a+2)T^{4}-(a-2)W^{2}=4, \quad \text{$T,W$ odd $>1$.}
\end{equation}

Since $a$ is odd, equation~\eqref{eq:(a+2)T^4-(a-2)W^2=4} falls under the scope of
Lemma~\ref{lem:L67} with $A=a+2$ and $B=a-2$. The equation $Az_{1}^{2}-Bz_{2}^{2}=4$
has $\left( a_{1}, b_{1} \right)= (1,1)$ as its
least solution in odd positive integers. Observe that $Aa_{1}^{2}-3=a-1$. 
Therefore, by Lemma~\ref{lem:L67}, we conclude the following.
If $a-1\not\in\rS$, then $T=1$ yields the only solution.
If $a-1=k^{2}$, then there is exactly one more 
solution, namely the one with $T=k$ (and $W=k^{2}+2$). 
The solution $T=1$ is rejected. The solution $t_{n}(a)=T^{2}=k^{2}$,
is also rejected, because $t_{1}(a)=a-1=k^{2}$ and for $a>2$, the
sequence $\left( t_{n}(a) \right)_{n \geq 0}$ is strictly increasing
(see Remark~\ref{rem:increasing}).
\end{proof}

\begin{proposition}
\label{prop:2n+1 n<>1(mod 3) a odd}
If $a$ is odd and $n\not\equiv 1\pmod{3}$, then $u_{2n+1}(a) \in p\rS$ is impossible.	
\end{proposition}

\begin{proof}
If $u_{2n+1}(a)\in p\rS$ then, by Lemma~\ref{lem:t_n or w_n must be square}, 
$w_{n}(a)\in\rS$ or $t_{n}(a)\in\rS$. Both alternatives are impossible; the first 
by Lemma~\ref{lem:w_n in S} 
and the second by Lemma~\ref{lem:t_n in S odd a}.
\end{proof}

\subsubsection{$a$ is even}
\label{subsubsec:a_even}

%For general even $a$, the present state of the art does not permit us to conclude
%that $(T,W)=(1,1)$ is the only positive solution of the equation
%\eqref{eq:(a+2)T^4-(a-2)W^2=4}.
%
We recall our assumption at the beginning of Section~\ref{sec:u_2n+1 in pS} that
$a \geq 3$ (hence $a \geq 4$ in this subsubsection) and $n \geq 2$. Under these
conditions and with our present knowledge, the best result concerning the relation
$t_{n}(a)\in\rS$ is Lemma~\ref{lem:t_n square,general a}. It
asserts that if $t_{n}(a)\in\rS$, then $n$ is necessarily odd and
either $(n,a)=(3,6)$ or $2n+1$ is a prime $\geq 19$ and $a\equiv 2\pmod{4}$.

Before proving Lemma~\ref{lem:t_n square,general a} we treat the values $n=3$
and $5$ separately, establishing the following result.

\begin{lemma}
\label{lem:t_5=square}
The only integer solutions $(x,y)$ with $y>0$ for $t_{3}(x)=y^{2}$ are
$(x,y)=(-1,1), (0,1)$, $(2,1), (6,13)$, and for $t_{5}(x)=y^{2}$ are $(x,y)=(-1,1), (1,1), (2,1)$.
Consequently, if $a>2$, then $t_{5}(a)\in\rS$ 
is impossible and $t_{3}(a)\in\rS$ is true only when $a=6$.
\end{lemma}

\begin{proof}
We have $t_{3}(x)=x^{3}-x^{2}-2x+1$. The elliptic curve defined by
$y^{2}=x^{3}-x^{2}-2x+1$ is of rank $1$. Applying the {\tt IntegralPoints}
function in {\sc magma}\cite{magma}\footnote{For the background of this function,
see \cite{StTz94} and/or \cite{GePeZi94}.}
we see that the only integer points with positive $y$-coordinate are those
stated in the lemma.

Next we consider the equation
\begin{equation}
\label{eq:t_5(x)=y^2}
t_{5}(x)=x^{5}-x^{4}-4x^{3}+3x^{2}+3x-1=y^{2}
\end{equation} 

For $x\leq -2$ we have $t_{5}(x)<0$. Also, $t_{5}(-1)=1, t_{5}(0)=-1$ and
$t_{5}(1)=1$. Therefore, in what follows we assume that $x\geq 2$.

Let $K=\bbQ(\theta)$ where $t_{5}(\theta)=0$. Actually, $K$ is the maximal real
subfield of the $11$-th cyclotomic field and $\theta=-\left( \zeta_{11}+\zeta_{11}^{-1} \right)$,
where $\zeta_{11}$ is a primitive eleventh root of unit. Thus, $K/\bbQ$ is a
Galois extension and 
\[
t_{5}(x)=\prod_{i=0}^{4} \left( x-\sigma^{i}(\theta) \right)
= \prod_{i=0}^{4} L_{i}(x) \quad
\text{with}\; L_{i}(x):=x-\sigma^{i}(\theta),
\]
where $\sigma$ is the generator of the Galois group, with $\sigma(\theta)=2-\theta^{2}$.
For $i=0,\ldots,4$ we put $\epsilon_{i+1}:=\sigma^{i}(\theta)$, so that
$\epsilon_{1},\ldots,\epsilon_{5}$ are the roots of $t_{5}(x)$. It turns out
that $\left\{ \epsilon_{1},\ldots,\epsilon_{4} \right\}$ is a set of fundamental
units\footnote{For instance, using PARI/GP, we find that the regulator is $1.63569\ldots$,
which equals the regulator of the unit subgroup generated by the $\epsilon_{i}$'s.}
%gp > nf=bnfinit(x^5+x^4-4*x^3-3*x^2+3*x+1,1);
%nf.reg
and we have $N \left( \epsilon_{i} \right)=+1$ for $i=1,\ldots,4$. By the
definition of the $\epsilon_{i}$'s, we have 
\begin{equation}
\label{eq:relations_of_epsilons}
\sigma \left( \epsilon_{j} \right)
=\epsilon_{j+1} \;\; \text{for}\; j=1,2,3 \quad\text{and}\;\;
\sigma \left( \epsilon_{4} \right)=\epsilon_{5}
= \left( \epsilon_{1}\epsilon_{2}\epsilon_{3}\epsilon_{4} \right)^{-1}.
\end{equation}

Let $N$ denote the norm function of the extension $K/\bbQ$. For $1\leq i<j\leq 5$,
we have $\left| N \left( \epsilon_{j}-\epsilon_{i} \right) \right|=11$. Therefore,
if $(x,y)$ is an integer
solution of \eqref{eq:t_5(x)=y^2}, then $x-\epsilon_{j}$ and $x-\epsilon_{i}$ are
relatively prime for any $i,j$ with $1\leq i<j\leq 5$. For, otherwise, any prime
divisor of them divides $11$, hence
$11 | y$ (in $\bbZ$). But then, $t_{5}(x)\equiv 0 \pmod{11^{2}}$ which is impossible,
as a direct computation shows. Thus, in the equation $L_{0}(x)\cdots L_{4}(x)
=t_{5}(x)=y^{2}$ the factors in 
the left-most side are relatively prime in pairs and this implies that
$L_{0}(x)= \pm \left( \prod_{j=1}^{4}\epsilon_{j}^{i_{j}} \right)\!\alpha_{0}^{2}$,
where each $i_{j} \in \{0,1\}$ and $\alpha_{0} \in K$ is an algebraic integer. 
For $i=1,\ldots,4$ we put $\alpha_{i}=\sigma^{i} \left( \alpha_{0} \right)$. Since 
$N \left( L_{0}(x) \right)=t_{5}(x)=y^{2}$ and $N \left( \epsilon_{i} \right)=+1$
for every $i$, the minus sign must be excluded.

We choose the embedding $^{*}:K\hookrightarrow\bbR$ with $\theta^{*}=-1.68250706\ldots$.
Then $\epsilon_{i}^{*}<0$ for $i=1,2$ and $\epsilon_{i}^{*}>0$ for $i=3,4$. We
also check that, for $x \geq 2$ we have 
$L_{0}(x)^{*}>L_{1}(x)^{*}>L_{3}(x)^{*}>L_{2}(x)^{*}>L_{4}(x)^{*}>0$.
In view of these observations we conclude that
\[
0<L_{0}(x)^{*}=
\left( \epsilon_{1}^{*} \right)^{i_{1}}
\left( \epsilon_{2}^{*} \right)^{i_{2}}
\left( \epsilon_{3}^{*} \right)^{i_{3}}
\left( \epsilon_{4}^{*} \right)^{i_{4}} \left( \alpha_{0}^{*} \right)^{2},
\]
implying that $i_{1}+i_{2} \equiv 0\pmod{2}$.

By \eqref{eq:relations_of_epsilons}, we compute 
$L_{1}(x)=\sigma \left( L_{0}(x) \right)=
\epsilon_{1}^{-i_{4}}\epsilon_{2}^{i_{1}-i_{4}}\epsilon_{3}^{i_{2}-i_{4}}
\epsilon_{4}^{i_{3}-i_{4}}\alpha_{1}^{2}$.
Using the embedding $^{*}$, as before, we conclude that $-i_{4} + \left( i_{1}-i_{4} \right)$
must be even, i.e. $i_{1}$ is even, hence $i_{1}=0$ and by our previous conclusion
that $i_{1}+i_{2}$ is even, we also have $i_{2}=0$. Proceeding in an exactly
similar way we obtain, from $L_{2}(x)^{*}>0$, \footnote{For this inequality to hold
it is necessary that $x \geq 2$. This explains why the solution with $x=1$ is
not found in the end.} that $i_{4}=0$, and from $L_{3}(x)^{*}>0$ that $i_{3}=0$. 
Therefore $L_{0}(x)=\alpha_{0}^{2}$ and, consequently, $L_{i}(x)=\alpha_{i}^{2}$
for $i=1,2,3,4$.

Now we consider the elliptic curve $E: L_{0}(X)L_{1}(X)L_{2}(X)=Y^{2}$ defined
over $K$ and we want to compute all points on $E$ with $\bbQ$-rational coordinate
$X$. This problem is solved by
a rather straightforward application of the simplest case of the \emph{Elliptic Chabauty} method
which uses the functions {\tt PseudoMordellWeilGroup} and {\tt Chabauty} of 
{\sc magma}\cite{magma}. We encourage the interested reader to see 
Example H130E24 in the on-line Magma handbook \cite{onlineMagma}.
In our case, it turns out that the only (finite) $K$-points $(X,Y)$ with $\bbQ$-rational $X$ are 
$(2,\pm(\theta-1))$. Therefore the only integer solution of \eqref{eq:t_5(x)=y^2} with $x\geq 2$
and $y>0$ is $(x,y)=(2,1)$.
\end{proof}

In the proof of Lemma~\ref{lem:t_n square,general a} we will use the following lemma.

\begin{lemma}
\label{lem:TogbeVoutierWalsh}
Assume that the equation
\begin{equation}
\label{eq:ax^4-by^2=1}
AX^{4}-BY^{2}=1,\quad \text{$A>1$ not a square, $B\ge 1$}
\end{equation}
has a solution $(X,Y)$ with $X,Y\ge 1$ and denote by $(x,y)$ the least such solution. 
Let 
\[
t:=Ax^{4}-1, \quad \tau:=\sqrt{t+1}+\sqrt{t}, \quad 
\tau^{2k+1}:=V_{2k+1}\sqrt{t+1}+W_{2k+1}\sqrt{t}.
\]

Equation~\eqref{eq:ax^4-by^2=1} has at most one positive solution $(X,Y)$
with $Y>y$. If such a solution exists and $Ax^{4}>2$, then $X^{2}=V_{p}$
for some prime $p\equiv 3\pmod{4}$.
\end{lemma}

\begin{proof}
According to Theorem~1.1 of \cite{Akht}, equation~\eqref{eq:ax^4-by^2=1} has at
most two solutions in positive integers. Assume that $(X,Y)$ is the
sole such solution with $Y>y$. Since $\left( X^{2},Y \right)$
is a solution of $Av^{2}-Bw^{2}=1$, we know that
$\left( X^{2},Y \right) = \left( V_{2n+1},W_{2n+1} \right)$
for some $n \geq 1$ -- see \cite[page~40]{TogVouWal2005}. Thus
$V_{2n+1}$ is a perfect square and $t>1$ by the assumption $Ax^{4}>2$.
Then Theorem~2.1 of \cite{TogVouWal2005} says that $n$ cannot be even, hence 
$2n+1=4m+3$ and now Corollary~2.1 of the same paper says that $4m+3$ must be a prime.
\end{proof}

\begin{lemma}
\label{lem:t_n square,general a}
If $a>2$ is even and $n>1$, then $t_{n}(a)$ can be a square only when $a \equiv 2 \pmod{4}$, 
$n$ is odd and either $(n,a)=(3,6)$ or else $n\ge 9$ and $2n+1$ is prime.
For such an $a$, there is at most one $n>1$ with $t_{n}(a)$ a square. 
\end{lemma}

\begin{proof}
Since $a^{2}>t_{2}(a)=a^{2}-a-1>(a-1)^{2}$ for $a>2$, $t_{2}(a)$ cannot be a square if $a>2$.
Also, under the assumptions concerning $a$ and $n$, it follows from Lemma~\ref{lem:t_5=square}
that if $n=3$ then $a=6$ and also that $t_{5}(a)$ cannot be a square.
Furthermore, $t_{4}(a)=a^{4}-a^{3}-3a^{2}+2a+1$ and applying the function
{\tt IntegralQuarticPoints} of {\sc magma}\cite{magma}
or the elementary method of \cite{Pou} (cf. footnote~\ref{foot:Poulakis}),
we see that $t_{4}(a)$ is not a square if $a>2$.
Therefore, in what follows we can assume that $n>5$.
Also recall that by hypothesis, we know that $a \geq 4$.

Assume that $t_{n}(a)=T^{2}$ and put $w_{n}(a)=W$. Since $a$ is even, both $T$
and $W$ are odd (by \eqref{eq:tn2-wn2-relationship} and \eqref{eq:gcd(t,w)}).
Also, $T,W>1$.
	
First let $a\equiv 2 \pmod{4}$.

We start by considering $a=6$. Here equation~\eqref{eq:tn2-wn2-relationship}
becomes $2T^{4}-W^{2}=1$. The only integral point $(T,W)$ with $T,W>1$ furnished
by the {\sc magma} function {\tt IntegralQuarticPoints} is $(T,W)=(13,239)$,
Hence $t_{n}(6)=T^{2}=13^{2}$, which holds only if $n=3$.

So $a \geq 10$ and we write \eqref{eq:tn2-wn2-relationship} with $b_{1}=1$ as 
\begin{equation}
\label{eq:(a+2)/4*T^4-(a-2)/4*W^2=1}
\frac{a+2}{4}\,T^{4} - \frac{a-2}{4}\,W^{2} = 1.
\end{equation}

First, we can ignore the case when $(a+2)/4$ is a perfect square. Indeed, assume
that $(a+2)/4=b^{2}$ and let $b=b_{1}b_{2}^{2}$, with $b_{1} \geq 1$ square-free.
Note that $(a-2)/4$ cannot be also a square. Therefore $(a-2)/4=d_{1}d_{2}^{2}$
with $d_{1}>1$ square-free. Then, equation~\eqref{eq:(a+2)/4*T^4-(a-2)/4*W^2=1}
becomes $b_{1}^{2} \left( b_{2}T \right)^{4} - d_{1} \left( d_{2}W \right)^{2}=1$.

In case that $b_{1}=1$, we have an equation of the form $X^{4}-d_{1}Y^{2}=1$. 
According to \cite{Cohn97}, this equation has at most one solution in positive
integers $X,Y$, except if $d_{1}=1785$, in which case there are exactly two positive
solutions, namely, $(X,Y)=(13,4)$ and $(239,1352)$. Regardless of whether $d_{1}=1785$
or not, the equation
$\left( b_{2}T \right)^{4} - d_{1} \left( d_{2}W \right)^{2}=1$ does
have the solution $(T,W)=(1,1)$. Hence this is the only solution if $d_{1} \ne 1785$.
Since $W>1$, this solution is rejected. When $d=1785$ we have $\left( b_{2}T, d_{2}W \right)
\in \{ (13,4),\, (239,1352) \}$. If $d_{2}W=4$, then, since $W$ is odd, we must have
$W=1$ which we reject. If $b_{2}T=239$ (prime), then, since $T>1$, we must
have $T=239$ and $b_{2}=1$, implying $b=b_{1}b_{2}^{2}=1$, hence $(a+2)/4=1$
which cannot hold since $a>2$. 

If $b_{1}>1$, then we have the equation 
$b_{1}^{2} \left( b_{2}T \right)^{4} - d_{1} \left( d_{2}W \right)^{2}=1$
with $b_{1}$ and $d_{1}$ both $>1$ and square-free. Theorem~1.2 of Bennett and
Walsh \cite{BW1999} says that there is at most one solution in positive integers
$(T,W)$, and since $(T,W)=(1,1)$ is such a solution here, it must be the only one
which, once again, we reject.

Thus, in the sequel, $(a+2)/4$ is not a perfect square.
We apply Lemma~\ref{lem:TogbeVoutierWalsh} to 
equation~\eqref{eq:(a+2)/4*T^4-(a-2)/4*W^2=1}
with $A=(a+2)/4>2$ (since $a \geq 10$), $B=(a-2)/4>1$, $(x,y)=(1,1)$, $t=(a-2)/4$ 
and $\tau=\sqrt{t+1}+\sqrt{t}=\left( \sqrt{a+2}+\sqrt{a-2} \right)/2$.
Then, in the notation of that lemma, $V_{1}=1=t_{0}(a)$ and $V_{3}=4t+1=a-1=t_{1}(a)$. 
Moreover, according to \cite{TogVouWal2005}, the following recurrence relation holds:
$V_{2k+3}=(4t+2)V_{2k+1}-V_{2k-1}$ for $k=1,2,\ldots$ and, in our present equation, 
$4t+2=a$. Also recall that the recurrence relation
satisfied by the sequence $\left( t_{n}(a) \right)_{n \geq 0}$ is $t_{n+1}(a)=a t_{n}(a)-t_{n-1}(a)$.
These observations lead to the conclusion that $V_{2k+1}=t_{k}(a)$ for every $k \geq 0$.
Finally, in the notation of Lemma~\ref{lem:TogbeVoutierWalsh}, $Ax^{4}=(a+2)/2>2$.
Therefore, according to the conclusion of the lemma,
the solution $(T,W)$ of equation~\eqref{eq:(a+2)/4*T^4-(a-2)/4*W^2=1}, being larger
than $(1,1)$, must satisfy $T^{2}=V_{p}$, where $p\equiv 3\pmod{4}$ is a prime number. 
But $T^{2}=t_{n}(a)=V_{2n+1}$, therefore $2n+1=p\equiv 3 \pmod{4}$. 
Thus, $n$ is odd and $2n+1$
is prime. The smallest odd $n>5$ which satisfies this condition is $9$. Further, if for a certain 
$a$ there exists a solution $(T,W)$ of \eqref{eq:(a+2)/4*T^4-(a-2)/4*W^2=1} larger than $(1,1)$
and $T^{2}=t_{n}(a)$, then no other solution to \eqref{eq:(a+2)/4*T^4-(a-2)/4*W^2=1} larger than 
$(1,1)$ exists. This implies that, if $t_{m}(a)$ is a square for some $m \ne n$, then, necessarily, 
$t_{m}(a)=T^{2}=t_{n}(a)$, which is impossible because the sequence 
$\left( t_{n}(a) \right)_{n \geq 0}$
is strictly increasing (see Remark~\ref{rem:increasing}).
	
Next, let $a\equiv 0\pmod{4}$. Now \eqref{eq:(a+2)T^4-(a-2)W^2=4} becomes
\begin{equation}
\label{eq:(a+2)/2*T^4-(a-2)/2*W^2=2}
\frac{a+2}{2}\,T^{4} - \frac{a-2}{2}\,W^{2} = 2.
\end{equation} 	
and the coefficients of $T^{4}$ and $W^{2}$ are odd.
By Theorem~1.4 of \cite{Yuan-Li}, the equation $Ax^{4}-By^{2}=2$ with positive
$A$ and $B$ has at most one solution in positive integers.
Applied to equation~\eqref{eq:(a+2)/2*T^4-(a-2)/2*W^2=2}, this says that,
the only solution in positive integers is $(T,W)=(1,1)$.
Since $T=t_{n}(a)>1$, this means that, in the present situation, $t_{n}(a)$ cannot
be a square.
\end{proof}

\pagebreak
\appendix 

\section{Interdependence of the results of this paper}
\label{app:interdependence}

The ``tree'' below shows the interdependence of the various numbered results in this paper.

T=Theorem \hspace*{5mm} P=Proposition \hspace*{5mm} L=Lemma

$R^{*}$ where $R\in\{T,P\}$ means that the proof of $R$ is heavily based on the
previous work of others concerning either equations $Ax^{4}-By^{2}=C\in \pm\{1,2,4\}$
or Lucas sequences.
\\[5pt]
\parbox[c]{40pt}{
\begin{tikzpicture}
\tikzset{level distance=50pt}
\Tree [.\labthm{res}{upp} \labthm{res}{low} ]
\end{tikzpicture}
}
means that ``res-upp'' is based on ``res-low''.

\begin{center}
\begin{tikzpicture}[sibling distance=1pt, scale=0.75]
\tikzset{level distance=80pt}
\Tree [.\labthm{T}{\ref{thm:at most one solution}}      % T 2.2 !
           [.\labthm{L*}{\ref{lem:RibMcDan}} ]  % L 4.1 !
           [.\labthm{L*}{\ref{lem:u_n(3)}} ]  % L 4.6 !
           [.\labthm{P*}{\ref{prop:mp}} ]  % P 4.7 !
           [.\labthm{P}{\ref{prop:u_2n in pS}}                   % P 5.1 !
                 [.\labthm{L}{\ref{lem:u_2n not in pS, a odd, n=0(mod3)}}  % L 5.2!
                   [.\labthm{L}{\ref{lem:2n a odd n=0 mod3 n even}}       % L  5.3!
                     [.\labthm{P*}{\ref{prop:mp}} ]                        % P 4.7 !
                   ]           
              [.\labthm{L}{\ref{lem:2n a odd n=0 mod3 n odd}}        % L 5.4!
                 [.\labthm{P*}{\ref{prop:mp}} ]                        % P 4.7 !
              ]          
                 ] 
                 [.\labthm{L}{\ref{lem:u_2n a and n even}}               % L 5.5!
                   [ .\labthm{P*}{\ref{prop:mp}} ]                                      % P 4.7!
                 ]          
                 [.\labthm{L}{\ref{lem:u_2n in pS => u_m in pS}}         % L 5.7 !
                   [.\labthm{L}{\ref{lem:u_2n in pS => u_{n} in pS}}   % L 5.6!
                     [.\labthm{P*}{\ref{prop:mp}} ]                                   % P 4.7 !                    
                   ]            
                 ]
           ] 
           [.\labthm{L}{\ref{lem:t_n or w_n must be square}} ]  % L 6.1 !
           [.\labthm{P}{\ref{prop:u_2n+1 not square with odd a}}  % P 6.2 !
            [.\labthm{L}{\ref{lem:2n+1 n=1(mod 3) a odd}}  % L 6.3  !   
              [.\labthm{P*}{\ref{prop:mp}} ]                        % P 4.7 !
            ]
             [.\labthm{P}{\ref{prop:2n+1 n<>1(mod 3) a odd}}    % P 6.9 !
               [.\labthm{L}{\ref{lem:t_n or w_n must be square}}  ]          % L 6.1 !
               [.\labthm{L*}{\ref{lem:w_n in S}}                                 % L 6.7 !
                 [.\labthm{L}{\ref{lem:Cohn66}}  ]           %  L 4.2 !
                 [.\labthm{L}{\ref{lem:x^2-(c^2-4)y^4=+4}}   % L 6.5 !
                   [.\labthm{L}{\ref{lem:x^2-(c^2-4)y^2=-4}}  ]   % L 6.4 !
                   [.\labthm{L}{\ref{lem:Cohn67}} ]       % L  4.3
                 ]
                 [.\labthm{L}{\ref{lem:x^2-(c^2-1)y^4=1}}        % L 6.6 !
                   [.\labthm{L}{\ref{lem:TVW2005}}  ]     % L 4.5!
                 ]
               ]          
              [.\labthm{L*}{\ref{lem:t_n in S odd a}}        % L 6.8 !
                 [.\labthm{L}{\ref{lem:L67}}  ]             % L 4.4 !
              ]
              ]     
             ]
           [.\labthm{L*}{\ref{lem:w_n in S}}                                 % L 6.7 !
                 [.\labthm{L}{\ref{lem:x^2-(c^2-4)y^4=+4}}   % L 6.5 !
                   [.\labthm{L}{\ref{lem:x^2-(c^2-4)y^2=-4}}  ]   % L 6.4 !
                   [.\labthm{L}{\ref{lem:Cohn67}}  ]     % L  4.3
                 ]
                 [.\labthm{L}{\ref{lem:x^2-(c^2-1)y^4=1}}        % L 6.6 !
                    [.\labthm{L}{\ref{lem:TVW2005}}  ]     % L 4.5!
                 ]
               ]
           [.\labthm{L*}{\ref{lem:t_n square,general a}}       % L 6.12 !
             [.\labthm{L}{\ref{lem:t_5=square}} ]            % L 6.10 !
             [.\labthm{L}{\ref{lem:TogbeVoutierWalsh}} ]  % L 6.11   !
           ]                                           
      ]    
\end{tikzpicture}  
\end{center}
The proof of $\mathrm{P}^*$\,\ref{prop:mp} is based on the proof of 
$\mathrm{L}^*$\,\ref{lem:RibMcDan} and $\mathrm{L}^*$\,\ref{lem:u_n(3)}. 
In the above ``tree'' we have omitted the numerous additional rows that this would require, due to 
space limitations.

\pagebreak
\section{``A Theorem about the Diophantine Equation 
$Ax^{2}-By^{4}=C$ ($C=1,2,4)$'' by W. Ljunggren}
\label{app:ljunggren}

\renewcommand{\theappthm}{\Roman{appthm}}
\setcounter{equation}{0}
\renewcommand{\theequation}{\arabic{equation}}

For the readers' convenience, we provide a translation here of the theorems in
Ljunggren's paper \cite{Ljun54}.
Although we do not use his Theorem~\ref{thm:IV}, we include it for completeness, 
also taking this opportunity to add an extra additional condition to Theorem~\ref{thm:IV} that is 
required by Ljunggren's proof. To aid references to his paper, we maintain
Ljunggren's notation, as well as his labelling of his theorems and equations, here.

Let $A$, $B$ and $C$ be rational, positive integers and $C=1$, $2$ or $4$. If $C$
is even, then $AB$ should be odd. Further, let $A$ be square-free, $AB$ not a
square and $C= 2$ for $A=1$. In an earlier work \cite{Ljun38} Ljunggren proved the
following theorems.
We note that, for Theorems I, II and III the requirement 
``$A$ be square-free, $AB$ not a square and $C=2$ for $A=1$'' 
can be very easily
replaced by \emph{``$AB$ not a square and $A$ can be a square only if $C=2$''}. 
In this paper we use only Theorems I, II and III, so we adopt this condition.

\begin{appthm}
\label{thm:I}
The equation
\begin{equation}
\label{eq:1}
Ax^{2}-By^{4}=C
\end{equation}
in the case of $C=1$ has at most one solution in positive, rational integers
$x$, $y$. If $B \equiv -1 \pmod{4}$, this also applies to $C=4$. In all other cases
there are at most two solutions.
\end{appthm}

One only needs to consider values of $A$, $B$ and $C$ for which an equation of
the form
\begin{equation}
\label{eq:2}
Az_{1}^{2}-Bz_{2}^{2}=C
\end{equation}
can be solved in natural numbers $z_{1}$ and $z_{2}$. Suppose that $\left( z_{1}, z_{2} \right)
=(a,b)$ is the smallest positive solution of \eqref{eq:2}.

\begin{appthm}
\label{thm:II}
If the quantity $\dfrac{4}{C}Bb^{2}+3$ is not a square, then $Ax^{2}-By^{4}=C$
has at most one solution in positive integers $x$ and $y$.
\end{appthm}

\begin{appthm}
\label{thm:III}
The equation $Ax^{2}-By^{4}=4$ has at most one solution in positive and relatively prime
integers $x$ and $y$.
\end{appthm}

\begin{appthm}
\label{thm:IV}
Suppose that $b$ in the minimal solution of \eqref{eq:2} is odd.
The equation $Ax^{4}-By^{2}=4$ has at most two solutions in positive integers
$x$ and $y$. If the quantity $Bb^{2}+1$ is not a square, then there
is at most one such solution. This also applies to solutions in coprime numbers
$x$ and $y$.
\end{appthm}

Note that we have added the extra condition here that $b$ must be odd. This extra
condition is required in Ljunggren's proof.

%%%%%%%%%%%%%%%%%%%%%%%%%%%%%%% BIBLIOGRAPHY
\bibliographystyle{amsplain}

%%%%%%%%%%%%%%%%%%%%%%%%%% End: References for the new approach
\end{document}